\RequirePackage[l2tabu,orthodox]{nag}
\documentclass[a4paper,11pt]{amsart}
\usepackage{amsmath,amsfonts,amsthm,amssymb,amsthm,stmaryrd}
\usepackage[all]{xy}

\usepackage[T1]{fontenc}
\usepackage[sc]{mathpazo}

\usepackage[margin=1.25in]{geometry}

\usepackage{eucal}
\usepackage{mathrsfs}

\renewcommand{\bar}[1]{\overline{#1}}
\renewcommand{\tilde}[1]{\widetilde{#1}}

\usepackage{hyperref}
\hypersetup{
  colorlinks=true,
  allcolors=blue
}

\newcommand{\sC}{\mathcal{C}}

\newcommand{\G}{\mathcal{G}}
\newcommand{\A}{\mathcal{A}}
\newcommand{\B}{\mathcal{B}}
\newcommand{\M}{\mathcal{M}}

\newcommand{\Z}{\mathcal{Z}}
\newcommand{\K}{\mathcal{K}}
\newcommand{\F}{\mathcal{F}}

\newcommand{\sO}{\mathcal{O}}
\newcommand{\sL}{\mathcal{L}}
\newcommand{\E}{\mathcal{E}}
\newcommand{\g}{\mathfrak{g}}

\newcommand{\C}{\mathbb{C}}
\newcommand{\R}{\mathbb{R}}
\newcommand{\bH}{\mathbf{H}}
\newcommand{\bE}{\mathbb{E}}

\newcommand{\bL}{\mathbb{L}}
\newcommand{\m}{\mathbf{m}}
\newcommand{\n}{\mathbf{n}}

\newtheorem{proposition}{Proposition}[section]
\newtheorem{lemma}[proposition]{Lemma}
\newtheorem{theorem}[proposition]{Theorem}
\newtheorem{corollary}[proposition]{Corollary}

\newtheorem{theoremintro}{Theorem}
\newtheorem*{theoremintro*}{Theorem}
\newtheorem*{corollaryintro*}{Corollary}

\numberwithin{equation}{section}

\DeclareMathOperator{\tr}{tr}

\DeclareMathOperator{\Aut}{Aut}
\DeclareMathOperator{\Lie}{Lie}
\DeclareMathOperator{\End}{End}

\DeclareMathOperator{\codim}{codim}

\title[The Quasi-projectivity of the Moduli Space of Higgs Bundles]{An Analytic Approach to the Quasi-projectivity of the Moduli Space of
  Higgs Bundles}
\author{Yue Fan}
\date{\today}
\address{Yue Fan, Department of Mathematics, University of Maryland,
  College Park, MD, 20742, USA}
\email{\texttt{yuefan@umd.edu}}

\begin{document}
\maketitle
\begin{abstract}
  The moduli space of Higgs bundles can be defined as a quotient of an
  infinite-dimensional space. Moreover, by the Kuranishi slice method, it is
  equipped with the structure of a normal complex space. In this paper, we will
  use analytic methods to show that the moduli space is quasi-projective. In
  fact, following Hausel's method, we will use the symplectic cut to construct
  a normal and projective compactification of the moduli space, and hence prove
  the quasi-projectivity. The main difference between this paper and Hausel's is
  that the smoothness of the moduli space is not assumed.
\end{abstract}
\tableofcontents

\section{Introduction}
Let $M$ be a closed Riemann surface with genus $\geq2$. A Higgs bundle
$(\E,\Phi)$ is a pair of a holomorphic vector bundle $\E\to M$ and a holomorphic
1-form taking values in $\End\E$. In \cite{Fan2020}, the moduli space $\M$ of
Higgs bundles is constructed as a quotient space of an infinite-dimensional
space. By the Kuranishi slice method, it is shown that $\M$ can be endowed with
the structure of a normal complex space. Moreover, it is shown that $\M$ is
canonically biholomorphic to the moduli space in the category of schemes, the
one constructed by Simpson and Nitsure using Geometric Invariant Theory (GIT)
(see \cite{Simpson1994a} and \cite{Nitsure1991}). As a consequence, $\M$ is a
quasi-projective variety. It is natural to ask whether the quasi-projectivity of
$\M$ can be proved directly by analytic methods without using the biholomorphism
just mentioned. The purpose of this paper is to give a positive answer.
Therefore, the first step toward this goal is to compactify the moduli space
$\M$. In \cite{Hausel1998}, Hausel used the symplectic cut to compactify the
moduli space $\M$ when it is smooth and the underlying smooth bundle of
$(\E,\Phi)$ is of rank 2. He further showed that the compactification is
projective and thus the quasi-projectivity of $\M$ follows. In this paper, we
will follow the same method to compactify the moduli space $\M$ and prove the
projectivity of the compactification. However, we will not impose the smoothness
conditions as Hausel did. Therefore, this paper is a generalization of Hausel's
results. It should be noted that Simpson compactified the moduli space using
algebro-geometric methods in \cite{Simpson1997}. However, the projectivity of
the compactification was not proved. In a recent paper \cite{Cataldo2020}, de
Cataldo followed Simpson's method and constructed a projective compactification
of the moduli space. It can be shown that our compactification is isomorphic to
de Cataldo's compactification.

To state the results and set up the notations, we fix a smooth Hermitian vector
bundle $E\to M$. For convenience, we assume that $E$ is of degree 0. This
condition is not essential. Since $\dim_\C M=1$, the space $\A$ of unitary
connections on $E$ parametrizes holomorphic structures on $E$. Let $\g_E\to M$
be the bundle of the skew-Hermitian endomorphisms of $E$, and set
$\sC=\A\times\Omega^{1,0}(\g_E^\C)$. The configuration space of Higgs bundles
(with the fixed underlying bundle $E$) is defined as
\begin{equation}
  \B=\{(A,\Phi)\in\sC\colon\bar{\partial}_A\Phi=0\}.
\end{equation}
The complex gauge group $\G^\C=\Aut(E)$ acts on $\B$ by 
\begin{equation}
  (\bar{\partial}_A,\Phi)\cdot g=(g^{-1}\circ\bar{\partial}_A\circ g,g^{-1}\Phi g),\qquad g\in\G^\C, (A,\Phi)\in\B.
\end{equation}
To define the moduli space $\M$, let us recall various stability conditions. A
Higgs bundle $(\E,\Phi)$ is semistable if $\mu(\F)\leq\mu(\E)$ for every
$\Phi$-invariant holomorphic subbundle $\F$ with $0\subsetneq\F\subsetneq\E$,
where $\mu(\F)$ means the slope of $\F$. If the equality $\mu(\F)=\mu(\E)$
cannot occur, then $(\E,\Phi)$ is stable. Finally, $(\E,\Phi)$ is polystable if
it is a direct sum of stable Higgs bundles of the same slope. Consequently,
there are $\G^\C$-invariant subspaces $\B^{ss}$, $\B^s$ and $\B^{ps}$ of $\B$
consisting of semistable, stable and polystable Higgs bundles, respectively.
Moreover, the polystability has a gauge-theoretic interpretation as follows.
Recall that $\sC$ is an infinite-dimensional affine hyperK\"ahler manifold that
is modeled on $\Omega^1(\g_E)\oplus\Omega^{1,0}(\g_E^\C)$ (see
\cite[\S6]{Hitchin1987b}). If we identify $\Omega^1(\g_E)$ with
$\Omega^{0,1}(\g_E^\C)$, an $L^2$-metric on $\sC$ is given by
\begin{equation}
  g(a,\eta;a,\eta)=\frac{2\sqrt{-1}}{4\pi^2}\int_M\tr(a^*a+\eta\eta^*),\qquad(a,\eta)\in\Omega^{0,1}(\g_E^\C)\oplus\Omega^{1,0}(\g_E^\C).
\end{equation}
Let $I$ be the complex structure given by the multiplication by $\sqrt{-1}$,
$\Omega_I$ its associated K\"ahler form, and $\G$ the subgroup of $\G^\C$
consisting of unitary gauge transformations. The $\G$-action on $\sC$ is
Hamiltonian with respect to the K\"ahler form $\Omega_I$, and the moment map is
given by
\begin{equation}
  \mu(A,\Phi)=\frac{1}{4\pi^2}(F_A+[\Phi,\Phi^*])\colon\sC\to\Omega^2(\g_E).
\end{equation}
Then, the Hitchin-Kobayashi correspondence states that a Higgs bundle
$(A,\Phi)\in\B$ is polystable if and only if $(A,\Phi)\cdot g$ satisfies
Hitchin's equation $\mu=0$ for some $g\in\G^C$. A stronger version of this
result states that the inclusion $\mu^{-1}(0)\cap\B\hookrightarrow\B^{ps}$
induces a homeomorphism $(\mu^{-1}(0)\cap\B)/\G\xrightarrow{\sim}\B^{ps}/\G^\C$,
where the inverse is induced by the retraction
$r\colon\B^{ss}\to\mu^{-1}(0)\cap\B$ defined by the Yang-Mills-Higgs flow (for
more details, see \cite{Wentworth2016} and \cite{Wilkin2008}).

By definition, the moduli space $\M$ is defined as the quotient $\B^{ps}/\G^\C$
equipped with the $C^\infty$-topology. In \cite{Fan2020}, it is shown that $\M$
is a normal complex space. One of the main results in this paper is that $\M$
admits a compactification.
\begin{theoremintro}\label{sec:introduction-compactify}
  There is a normal compact complex space $\bar{\M}$ in which the moduli space
  $\M$ embeds as an open dense subset. Moreover, the complement
  $Z=\bar{\M}\setminus\M$ is a closed complex subspace of pure codimension 1.
\end{theoremintro}
As a consequence, the quasi-projectivity of $\M$ follows if we can show that the
compactification $\bar{\M}$ is projective. Therefore, we need to construct an
ample line bundle on $\bar{\M}$. To construct such a line bundle, we need a
descent lemma for vector bundles. In \cite{Drezet1989}, Drezet and Narasimhan
proved a descent lemma for good quotients of algebraic varieties. A natural
analogue of good quotients in our settings is the quotient map
$\pi\colon\B^{ss}\to\B^{ss}\sslash\G^\C$, where $\B^{ss}\sslash\G^\C$ is the
quotient space of $\B^{ss}$ by the $S$-equivalence relation of Higgs bundles.
Recall that two Higgs bundles are $S$-equivalent if the graded objects of their
Seshadri filtrations are isomorphic. Heuristically, we think of
$\pi\colon\B^{ss}\to\B^{ss}\sslash\G^\C$ as an infinite-dimensional GIT quotient
and naturally expect that its properties are similar to those of good quotients
of algebraic varieties. To justify this heuristic thinking, we will first prove
in Section~\ref{sec:s-equiv-class} that the inclusion
$\B^{ps}\hookrightarrow\B^{ss}$ induces a homeomorphism
$\M\to\B^{ss}\sslash\G^\C$, and hence will routinely identify $\M$ with
$\B^{ss}\sslash\G^\C$. Then, we will show the following.
\begin{theoremintro}\label{sec:introduction-infi-dim-GIT}
  The quotient map $\pi\colon\B^{ss}\to\M$ satisfies the
  following properties:
  \begin{enumerate}
  \item $\pi$ identifies $\G^\C$-orbits whose closures in $\B^{ss}$ intersect.
    
  \item Every fiber of $\pi$ contains a unique $\G^\C$-orbit that is closed in
    $\B^{ss}$. Moreover, a $\G^\C$-orbit is closed in $\B^{ss}$ if and only if
    it contains a polystable Higgs bundle.
    
  \item $\sO_\M=\pi_*\sO_{\B^{ss}}^{\G^\C}$. In other words, if $U$ is an open
    subset of $\M$, the map $\sO_\M(U)\to\sO_{\B^{ss}}(\pi^{-1}(U))^{\G^\C}$
    given by $f\mapsto \pi^*f$ is a bijection.
    
  \item $\pi$ is a categorical quotient in the sense that every
    $\G^\C$-invariant holomorphic map from $\B^{ss}$ into a complex space
    factors through the quotient map $\pi\colon\B^{ss}\to\B^{ss}\sslash\G^\C$.
  \end{enumerate}
\end{theoremintro}
To make sense of $(3)$ in Theorem~\ref{sec:introduction-infi-dim-GIT}, we equip
the space $\B$ with a naive structure sheaf by restricting the sheaf of
$I$-holomorphic functions on $\sC$ to $\B$. Moreover, $\sO_\M$ denotes the
structure sheaf of the moduli space $\M$.

Generalizing the descent lemma for vector bundles in \cite{Drezet1989}, we
will prove the following (cf. \cite[Lemma 2.13]{Sjamaar1995}).
\begin{theoremintro}\label{sec:introduction-descent-lemma}
  Let $\bE\to\B^{ss}$ be a holomorphic $\G^\C$-bundle. Suppose that the
  stabilizer $\G^\C_{(A,\Phi)}$ acts trivially on the fiber $\bE_{(A,\Phi)}$ for
  every $(A,\Phi)\in\mu^{-1}(0)$. Then, there is a holomorphic
  vector bundle $E$ over $\M$ such that $\pi^*E=\bE$. Moreover,
  $\sO(E)=\pi_*\sO(\bE)^{\G^\C}$, where $\sO(E)$ and $\sO(\bE)$ are sheaves of
  holomorphic sections of $E$ and $\bE$, respectively.
\end{theoremintro}

Now we are able to construct an ample line bundle on $\M$ as follows. Recall
that $\A$ is an infinite-dimensional K\"ahler manifold that is modeled on
$\Omega^1(\g_E)$ (see \cite[p.587]{Atiyah1983}). In \cite{Donaldson1987a},
Donaldson constructed a holomorphic line bundle on $\A$ together with a
Hermitian metric whose curvature is a multiple of K\"ahler form on $\A$ (also
see \cite{Quillen1985}). Moreover, the $\G^\C$-action on $\A$ lifts to this line
bundle. By pulling back this line bundle to $\sC$ by the projection map
$\sC\to\A$, we obtain an $I$-holomorphic line bundle $\bL\to\sC$, and the
$\G^\C$-action on $\sC$ lifts to $\bL$. By slightly modifying the pullback
Hermitian metric, we are able to show that the curvature of the resulting
Hermitian metric $h$ on $\bL$ is $-2\pi\sqrt{-1}\Omega_I$. Then, the
projectivity of the compactification $\bar{\M}$ is shown in the following
result.
\begin{theoremintro}\label{sec:introduction-positivity}\
  \begin{enumerate}
  \item The restriction of the line bundle $\bL\to\sC$ to $\B^{ss}$ descends to $\M$
    and defines a line bundle $\sL\to\M$.
    
  \item $\sL$ extends to a line bundle $\bar{\sL}$ on $\bar{\M}$.
    
  \item $\bar{\sL}$ is ample.
  \end{enumerate}
  Therefore, $\bar{\M}$ is projective, and hence $\M$ is quasi-projective.
\end{theoremintro}
In the proof of $(1)$ in Theorem~\ref{sec:introduction-positivity}, a byproduct
is that the moduli space $\M$ has a weak K\"ahler metric. More precisely, we
recall that $\M$ admits an orbit type stratification such that each stratum $Q$
is a complex submanifold of $\M$ together with a K\"ahler form $\omega_Q$ (see
Section~\ref{sec:orbit-type-strat}). A weak K\"ahler metric on $\M$ is a family
of continuous stratum-wise strictly plurisubharmonic functions $\rho_i\colon
U_i\to\R$ such that $\{U_i\}$ is an open covering of $\M$ and that
$\rho_i-\rho_j=\Re(f_{ij})$ for some holomorphic function
$f_{ij}\in\sO_{\M}(U_i\cap U_j)$. Here, a continuous stratum-wise strictly
plurisubharmonic function is a continuous function that is smooth and strictly
plurisubharmonic along every stratum $Q$ in the orbit type stratification of
$\M$. Note that stratum-wise strictly plurisubharmonic functions are not
necessarily strictly plurisubharmonic. If each $\rho_i$ can be chosen to be
strictly plurisubharmonic, then $\{\rho_i\colon U_i\to\R\}$ defines a (strong)
K\"ahler metric on $\M$. (see \cite{Heinzner1994a} for more details on strictly
plurisubharmonic functions). Finally, since
$\sqrt{-1}\partial\bar{\partial}(\rho_i|_Q)$ patches together, the K\"ahler
metric on $\M$ restricts to $Q$. Then, our last result is the following.
\begin{theoremintro}\label{sec:introduction-singular-Kahler-metric}
  The moduli space $\M$ admits a weak K\"ahler metric whose restriction to
  each stratum $Q$ in the orbit type stratification of $\M$ is the K\"ahler form
  $\omega_Q$.
\end{theoremintro}

In \cite{Fan2020a}, it is shown that each stratum $Q$ in the orbit type
stratification of $\M$ also admits a complex symplectic form such that they glue
together to define a complex Poisson bracket on the structure sheaf $\sO_\M$ of
$\M$. In fact, each stratum $Q$ is a hyperK\"ahler manifold. The result in
\cite{Fan2020a} and Theorem~\ref{sec:introduction-singular-Kahler-metric} show
that the stratum-wise defined hyperK\"ahler structure on $\M$ can be extended to
two global holomorphic objects, a complex Poisson bracket and a weak K\"ahler
metric. Finally, we should remark that we are unable to prove that the weak
K\"ahler metric on $\M$ is strong, although it is highly likely.

Now we describe the structures of the paper and the ideas behind the proofs
of the main theorems. The key tools in the proof of
Theorem~\ref{sec:introduction-infi-dim-GIT} are a local slice theorem for the
$\G^\C$-action and the retraction $r\colon\B^{ss}\to\mu^{-1}(0)\cap\B$ defined
by the Yang-Mills-Higgs flow. We will prove
Theorem~\ref{sec:introduction-infi-dim-GIT} in
Section~\ref{sec:infin-dimens-git}. To prove
Theorem~\ref{sec:introduction-descent-lemma}, we will first prove a descent
lemma for analytic Hilbert quotients of complex spaces. By definition, if $G$ is
a complex Lie group, then an analytic Hilbert quotient of a holomorphic
$G$-space $X$ is a $G$-invariant surjective holomorphic map $\pi\colon X\to Z$
such that inverse images of Stein subspaces are Stein, and $\sO_Z=\pi_*\sO_X^G$.
This notion is an analytic analogue of good quotients of algebraic varieties
(see \cite{Heinzner1998}). The proof of
Theorem~\ref{sec:introduction-descent-lemma} is an adaptation of Drezet and
Narasimhan's argument in \cite{Drezet1989}. Then, this result will be applied to
Kuranishi local models that are used to construct the moduli space $\M$, since
Kuranishi local models are analytic Hilbert quotients of Kuranishi spaces. In
this way, we can show that every point in $\B^{ss}$ admits an open neighborhood
that is saturated with respect to the quotient map $\pi\colon\B^{ss}\to\M$ and
in which the vector bundle $\bE$ in question is trivial. This shows that $\bE$
descends to $\M$. These results will be proved in
Section~\ref{sec:desc-lemm-vect}. After
Theorem~\ref{sec:introduction-infi-dim-GIT} and
\ref{sec:introduction-descent-lemma} are proved, we are ready to prove
Theorem~\ref{sec:introduction-singular-Kahler-metric} and $(1)$ in
Theorem~\ref{sec:introduction-positivity}. By verifying the hypothesis in
Theorem~\ref{sec:introduction-descent-lemma} for the line bundle
$\bL|_{\B^{ss}}$, we can easily show that it defines a line bundle $\sL\to\M$.
To show Theorem~\ref{sec:introduction-singular-Kahler-metric}, we may choose an
open covering $\{U_i\}$ of $\M$ such that $\sL$ is trivial over each
$\pi^{-1}(U_i)$. Then, we choose a holomorphic section $s_i$ of $\bL$ over each
$\pi^{-1}(U_i)$ that is $\G^\C$-equivariant and nowhere vanishing. Then, we
consider the functions
\begin{equation}
  u_i=-\frac{1}{2\pi}\log|s_i|_h^2,
\end{equation}
where $h$ is the Hermitian metric on $\bL$. Since it is $\G$-invariant, its
restriction to $\pi^{-1}(U_i)\cap\mu^{-1}(0)$ defines a continuous map
$u_{i,0}\colon U_i\to\R$. It will be shown that the restriction of each
$u_{i,0}$ to a stratum $Q$ is smooth and a K\"ahler potential for the K\"ahler
form $\omega_Q$ on $Q$. In this way, we obtain a family of continuous
stratum-wise strictly plurisubharmonic functions $u_{i,0}\colon U_i\to\R$ such
that $\{U_i\}$ covers $\M$. Then, the normality of $\M$ and the fact that
$\codim_x(\M\setminus\M^s)\geq2$ for all $x\in\M\setminus\M^s$ show that
$\{u_{i,0}\colon U_i\to\R\}$ defines a weak K\"ahler metric. These results will
be proved in Section~\ref{sec:kahler-metric-moduli}.

Then, we will prove Theorem~\ref{sec:introduction-compactify} and the rest of
the statements in Theorem~\ref{sec:introduction-positivity} in
Section~\ref{sec:proj-comp}. Following Hausel's strategy in \cite{Hausel1998},
we will use the symplectic cut to compactify $\M$. Recall that $\M$ admits a
holomorphic $\C^*$-action. Moreover, the induced $U(1)$-action is stratum-wise
Hamiltonian. More precisely, the restriction of the $\G$-invariant map
\begin{equation}
  f(A,\Phi)=-\frac{1}{4\pi^2}\frac{1}{2}\|\Phi\|_{L^2}^2\colon\sC\to\R
\end{equation}
to $\mu^{-1}(0)$ defines a continuous map $f\colon\M\to\R$. When restricted to a
stratum $Q$, $f|_Q$ is smooth and a moment map for the induced $U(1)$-action
with respect to the K\"ahler form $\omega_Q$ on $Q$. In this sense, $f$ is a
stratum-wise moment map on $\M$. Then, we consider the direct product
$\M\times\C$. If we let $\C^*$ act on $\C$ by multiplication, $\M\times\C$
admits a diagonal $\C^*$-action. The induced $U(1)$-action is also stratum-wise
Hamiltonian. Here, the stratification of $\M\times\C$ is given by the disjoint
union of $Q\times\C$, where $Q$ ranges in the orbit type stratification of $\M$.
Moreover, the stratum-wise moment map on $\M\times\C$ is given by
\begin{equation}
  \tilde{f}=f-\frac{1}{2}\|\cdot\|^2.
\end{equation}
By \cite[Theorem 8.1]{Hitchin1987b} or \cite[Theorem 2.15]{Wentworth2016}, the
Hitchin fibration $h$ is proper, and hence the nilpotent cone $h^{-1}(0)$ is
compact. Therefore, we are able to choose a level $c<0$ such that
$h^{-1}(0)\subset f^{-1}[0,c)$. Then the symplectic cut of $\M$ at the level
$c$ is defined as the singular symplectic quotient $\tilde{f}^{-1}(c)/U(1)$,
and it should be a compactification of $\M$. Here, the rough idea is that the
subspace $f^{-1}[0,c]$ is compact by the properness of $f$ (see
\cite[Proposition 7.1]{Hitchin1987b}). Moreover, if a Higgs bundle is away from
$ f^{-1}[0,c]$, following its $\C^*$-orbit, it ``flows'' into $ f^{-1}[0,c]$,
since the $0$-limit of the $\C^*$-action on a Higgs bundle always exists, and
hence the limiting point is a $\C^*$-fixed point and is contained in the nilpotent
cone. Therefore, the moduli space $\M$ should be ``contained in''
$\tilde{ f}^{-1}(c)/U(1)$, which is compact because of the properness of $ f$.

To carry out this idea rigorously, we first need to equip
$\tilde{ f}^{-1}(c)/U(1)$ with the structure of a complex space. Let
$(\M\times\C)^{ss}$ be the subspace of semistable points in $\M\times\C$
determined by the stratum-wise moment map $\tilde{ f}-c$. More precisely, it
consists of points in $\M\times\C$ whose $\C^*$-orbit closures intersect
$\tilde{ f}^{-1}(c)$. To show that the analytic Hilbert quotient of
$(\M\times\C)^{ss}$ by $\C^*$ exists, we run into a technical
difficulty. Since we are unable to prove that the K\"ahler metric on $\M$ is a
strong one, we cannot directly apply the analytic GIT developed by Heinzner and
Loose in \cite{Heinzner1994} and must take a detour. To motivate the following
detour, let us recall that a complex reductive Lie group acts properly at a
point if and only if its stabilizer at that point is finite, provided that a
local slice theorem is available around that point. Since the $\C^*$-stabilizers
are finite away from the nilpotent cone $h^{-1}(0)$, it is reasonable to expect
that the $\C^*$-action acts properly away from the nilpotent cone. Hence, we
consider the $\C^*$-invariant open subset
$W=(\M\times\C)\setminus(h^{-1}(0)\times\{0\})$. By the properness of the
Hitchin fibration $h$, we can show that the $\C^*$-action on $W$ is proper, and
hence the analytic Hilbert quotient of $W$ by $\C^*$ exists. Moreover, $W/\C^*$
is a geometric quotient. Then, we use the properness of $h$ and $ f$ to show
that $W=(\M\times\C)^{ss}=\C^*\tilde{ f}^{-1}(c)$. It then follows that the
inclusion $\tilde{ f}^{-1}(c)\hookrightarrow W$ induces a homeomorphism
$\tilde{ f}^{-1}(c)/U(1)\to W/\C^*$. Now, note that $W$ can be written as a
disjoint union
\begin{equation}
  W=(\M\setminus h^{-1}(0)\times\{0\})\cup(\M\times\C^*).
\end{equation}
We will show that the quotient $(\M\times\C^*)/\C^*$ is biholomorphic to the
moduli space $\M$, and therefore $\bar{\M}=W/\C^*$ is a compactification of
$\M$. To show the rest of the statements in
Theorem~\ref{sec:introduction-positivity}, we pullback the line bundle
$\sL\to\M$ to $\M\times\C$ by the projection map $\M\times\C\to\M$ to obtain a
line bundle $\sL_\C\to\M\times\C$. By slightly modifying the Hermitian metric
$h$ on $\sL_\C$, we can easily show that the resulting Hermitian metric, again
denoted by $h$, is smooth along each stratum $Q\times\C$, and the curvature is
$-2\pi\sqrt{-1}\omega_{Q\times\C}$, where $\omega_{Q\times\C}$ is the product
K\"ahler metric on $Q\times\C$. By the descent lemma for the analytic Hilbert
quotients, the restriction of $\sL_\C$ to $W$ induces a line bundle
$\bar{\sL}\to\bar{\M}$ such that the restriction of $\bar{\sL}$ to $(\M\times\C^*)/\C^*$
is isomorphic to $\sL\to\M$. In this sense, the line bundle $\sL\to\M$ extends
to the line bundle $\bar{\sL}\to\bar{\M}$. Moreover, the Hermitian metric $h$ on
$\sL_\C$ also induces a Hermitian metric $\bar{h}$ on $\bar{\sL}$. Then, we will use
Popovici's bigness criterion (see \cite[Theorem 1.3]{Popovici2008}) to show that
the restriction of $\bar{\sL}$ to any irreducible closed complex subspace (not
reduced to a point) of $\bar{\M}$ is big. Then, the ampleness of $\bar{\sL}$ follows
from a theorem of Grauert (see \cite{Grauert1962}): a line bundle over a compact
complex space is ample if its restriction to any irreducible closed complex
subspace (not reduced to a point) admits a nontrivial holomorphic section that
vanishes somewhere on that subspace. These results will be proved in
Section~\ref{sec:proj-comp}.

Finally, we remark that all the complex spaces in this paper are assumed to be
reduced. Moreover, if necessary, we will work with the $L_k^2$-topology on $\sC$
and $L_{k+1}^2$-topology on $\G^\C$, where $k>1$ is fixed. By \cite[Corollary
3.13]{Fan2020a}, the topology of $\M$ does not depend on the choice of $k$.

\vspace{0.5cm}
\noindent\textbf{Acknowledgments}. This paper is part of my Ph.D. thesis. I
would like to thank my advisor, Professor Richard Wentworth, for suggesting this
problem and his generous support and guidance. I also thank Reyer Sjamaar,
Daniel Greb, and Ruadhai Dervan for helpful discussions.

\section{Preliminaries}

\subsection{Construction of the moduli space}\label{sec:constr-moduli-space}
In this section, we review the results in \cite{Fan2020}. Recall that every Higgs bundle
$(A,\Phi)\in\B$ defines a deformation complex
\begin{equation}
  C_{\mu_\C}(A,\Phi)\colon\qquad\Omega^0(\g_E^\C)\xrightarrow{D''}\Omega^{0,1}(\g_E^\C)\oplus\Omega^{1,0}(\g_E^\C)\xrightarrow{D''}\Omega^{1,1}(\g_E^\C),
\end{equation}
where $D''=\bar{\partial}_A+\Phi$. If the Higgs bundle $(A,\Phi)$ is understood,
we will simply write $C_{\mu_\C}$ instead of $C_{\mu_\C}(A,\Phi)$.
\begin{proposition}[{\cite[\S1]{Simpson1992} and \cite[\S10]{Simpson1994a}}]
  The complex $C_{\mu_\C}$ is an elliptic complex. Moreover, the formal
  $L^2$-adjoint $(D'')^*$ satisfies the K\"ahler identities
  \begin{equation}
    (D'')^*=-i[*,D'],\qquad(D')^*=+i[*,D''],
  \end{equation}
  where $D'=\partial_A+\Phi^*$ and $*$ is the Hodge star.
\end{proposition}
Let $\bH^1$ denote the harmonic space for the cohomology $H^1(C_{\mu_\C})$.
Then, we review the Kuranishi slice method that is used to construct the moduli
space. Fix $(A,\Phi)\in\m^{-1}(0)$ with $\G$-stabilizer $K$, where
$\m=(\mu,\mu_\C)$ and $\mu_\C(A,\Phi)=\frac{1}{4\pi^2}\bar{\partial}_A\Phi$. As
a consequence, $K^\C$ is the $\G^\C$-stabilizer at $(A,\Phi)$. Note that since
the $\G$-action on $\sC$ is proper, $K$ is compact. Then, there are an open ball
in the $L^2$-topology around 0 in $\bH^1$ and a Kuranishi map $\theta\colon
B\to\sC$. Its associated Kuranishi space $\Z$ is defined as
$\Z=\theta^{-1}(\B)$. It can be shown that $\Z$ is a closed complex subspace of
$B$. Since $\B^{ss}$ is open in $\B$, if $B$ is sufficiently small, then we
obtain a restriction $\theta\colon\Z\to\B^{ss}$. We list some of its
properties that will be used later.
\begin{proposition}\ \label{sec:constr-moduli-space-kura-map-prop}
  \begin{enumerate}
  \item $\theta\colon B\to\sC$ is $I$-holomorphic, and $\theta(0)=(A,\Phi)$.
    
  \item The derivative of $\theta\colon B\to\sC$ at $0$ is the inclusion map
    $\bH^1\hookrightarrow T_{(A,\Phi)}\sC$.

  \item $\theta\colon B\to\sC$ extends to a $K^\C$-equivariant holomorphic map
    $\theta\colon BK^\C\to\sC$.
    
  \item If $B$ is sufficiently small, then $\theta$ preserves the stabilizers in
    the sense that $(K^\C)_x=(\G^\C)_{\theta(x)}$ for any $x\in\Z$.
    
  \item If $x\in\Z$ has a closed $K^\C$-orbit in $\bH^1$, then $\theta(x)$ is a polystable
    Higgs bundle.
  \end{enumerate}
\end{proposition}
Moreover, there is a local slice theorem for the $\G^\C$-action.
\begin{proposition}\label{sec:constr-moduli-space-local-slice}
  If $B$ is sufficiently small, the map $\Z
  K^\C\times_{K^\C}\G^\C\to\B^{ss}$ given by $[x,g]\mapsto\theta(x)g$ is a
  homeomorphism onto an open neighborhood of $(A,\Phi)$ in
  $\B^{ss}$.
\end{proposition}
It is shown that $\Z K^\C$ is a closed complex subspace of $BK^\C$. Moreover, it
admits an analytic Hilbert quotient (see \cite{Heinzner1998} and
\cite{Heinzner1994}) $\pi\colon\Z K^\C\to\Z K^\C\sslash K^\C$. In other
words, $\pi$ is a surjective $K^\C$-invariant map such that it is Stein in the
sense that preimages of Stein subspaces are Stein. Moreover, the structure sheaf
of $\Z K^\C\sslash K^\C$ is given by $\pi_*\sO_{\Z K^\C}^{K^\C}$. As a
topological space, $\Z K^\C\sslash K^\C$ is defined as a quotient space by the
equivalence relation that $x\sim y$ if the closures of $xK^\C$ and $yK^\C$
intersect. Then, it is shown that $\theta$ induces a well-defined map
$\varphi\colon\Z K^\C\sslash K^\C\to\M$ such that $\varphi[x]=[r\theta(x)]$ for
any $x\in\Z$. We also call $\varphi$ a Kuranishi map. Here,
$r\colon\B^{ss}\to\m^{-1}(0)$ is the retraction defined by the negative gradient
flow of the Yang-Mills-Higgs functional. By \cite[Theorem 1.4]{Wilkin2008},
$r(A,\Phi)$ is isomorphic to $Gr(A,\Phi)$, where $Gr(A,\Phi)$ is the graded
object associated with the Seshadri filtration of $(A,\Phi)$. Moreover,
$\varphi$ is a homeomorphism onto an open neighborhood of $[A,\Phi]$. Finally,
there is a unique structure of a normal complex space on $\B^{ps}/\G^\C$ such that
$\varphi$ is a biholomorphism onto its image.

More can be said about the singularities in the moduli space. Note that $\bH^1$
admits a linear hyperK\"ahler structure that is induced by the inclusion
$\bH^1\hookrightarrow T_{(A,\Phi)}\sC$. Let $\omega_0$ and $\omega_{0,\C}$
denote the K\"ahler form associated with the complex structure $I$ and the
complex symplectic form associated with the other complex structures,
respectively. The $\G^\C$-stabilizer $K^\C$ acts linearly on $\bH^1$ and
preserves $\omega_{0,\C}$. As a consequence, there is a canonical complex moment
map $\nu_{0,\C}\colon\bH^1\to\bH^2(C_{\mu_\C})$ given by
$\nu_{0,\C}(x)=\frac{1}{2}H[x,x]$, where
$H\colon\Omega^{1,1}(\g_E^\C)\to\bH^2(C_{\mu_\C})$ is the harmonic projection.
Then, $\nu_{0,\C}^{-1}(0)$, as an affine variety, admits a GIT quotient
$\pi\colon\nu_{0,\C}^{-1}(0)\to\nu_{0,\C}^{-1}(0)\sslash K^\C$. In fact, it can
be realized as a singular hyperK\"ahler quotient as follows. The $\G$-stabilizer
$K$ acts linearly on $\bH^1$ and preserves the K\"ahler form $\omega_0$. As a
consequence, the $K$-action on $\bH^1$ admits a unique moment map $\nu_0$ such
that $\nu_0(0)=0$. Then, $\n=(\nu_0,\nu_{0,\C})$ can be regarded as a
hyperK\"aher moment map. By \cite{Heinzner1994}, the inclusion
$\n^{-1}(0)\hookrightarrow\nu_{0,\C}^{-1}(0)$ induces a homeomorphism
$\n^{-1}(0)/K\xrightarrow{\sim}\nu_{0,\C}^{-1}(0)\sslash K^\C$. Finally, it can
be shown that $\Z=B\cap\nu_{0,C}^{-1}(0)$, and $\Z K^\C\sslash K^\C$ is an open
neighborhood of $[0]$ in $\nu_{0,C}^{-1}(0)\sslash K^\C$. In summary, $[A,\Phi]$
admits an open neighborhood that is biholomorphic to an open neighborhood of
$[0]$ in $\nu_{0,\C}^{-1}(0)\sslash K^\C$.

\subsection{The orbit type stratification}\label{sec:orbit-type-strat}
In this section, we first review the results in \cite{Fan2020a}, and then prove
some technical results that will be used later.

Let $K$ be a $\G$-stabilizer at some Higgs bundle in $\m^{-1}(0)$ and $(K)$ the
conjugacy class of $K$ in $\G$. The subspace
\begin{equation}
  \m^{-1}(0)_{(K)}=\{(A,\Phi)\in\m^{-1}(0)\colon \G_{(A,\Phi)}\in(H)\}
\end{equation}
is $\G$-invariant. The orbit type stratification of the singular hyperK\"ahler
quotient $\m^{-1}(0)/\G$ is defined as
\begin{equation}
  \m^{-1}(0)/\G=\coprod_{(H)}\text{ components of }\m^{-1}(0)_{(H)}/\G.
\end{equation}
Similarly, if $L$ is a $\G^\C$-stabilizer at some Higgs bundle in $\B^{ps}$, and
$(L)$ denotes the conjugacy class of $L$ in $\G^\C$, then the subspace
\begin{equation}
  \B^{ps}_{(L)}=\{(A,\Phi)\in\B^{ps}\colon (\G^\C)_{(A,\Phi)}\in(L)\}
\end{equation}
is $\G^\C$-invariant. The orbit type stratification of $\M$ is defined as
\begin{equation}
  \M=\coprod_{(L)}\text{ components of }\B^{ps}_{(L)}/\G^\C.
\end{equation}
It can be proved that the subgroup $L$ appearing in the orbit type
stratification of $\M$ is always equal to $K^\C$ for some compact subgroup $K$
of $\G$. Then, we have the following results.
\begin{proposition}\ \label{sec:orbit-type-strat-main-results}
  \begin{enumerate}
  \item Every stratum $Q$ in the orbit type stratification of $\m^{-1}(0)/\G$ is a
    locally closed smooth manifold, and $\pi^{-1}(Q)$ is a smooth submanifold of
    $\sC$ such that the restriction $\pi\colon\pi^{-1}(Q)\to Q$ is a smooth
    submersion. Moreover, the restriction of the hyperK\"ahler structure from
    $\sC$ to $\pi^{-1}(Q)$ descends to $Q$.
    
  \item Every stratum $Q$ in the orbit type stratification of $\M$ is a locally closed
    complex submanifold of $\M$, and $\pi^{-1}(Q)$ is a complex submanifold of
    $\sC$ with respect to the complex structure $I$ such that the restriction
    $\pi\colon\pi^{-1}(Q)\to Q$ is a holomorphic submersion. This decomposition is
    a complex Whitney stratification.
    
  \item If $Q$ is a stratum determined by the orbit type $(K^\C)$, and
    $[A,\Phi]\in Q$, then the tangent space of $Q$ at $[A,\Phi]$ can be
    identified with $(\bH^1)^{K^\C}$.
  \end{enumerate}
\end{proposition}
Moreover, the Hitchin-Kobayashi correspondence
$i\colon\m^{-1}(0)/\G\xrightarrow{\sim}\B^{ps}/\G^\C$ preserves the
stratifications in the following sense.
\begin{proposition}\label{sec:orbit-type-strat-Kob-Hit-preser}
  If $Q$ is a stratum  in the orbit type stratification of $\m^{-1}(0)/\G$, then
  $i(Q)$ is a stratum in the orbit type stratification of $\M$, and the
  restriction $i\colon Q\to i(Q)$ is a biholomorphism with respect to the
  complex structure $I_Q$ on $Q$ coming from $\sC$ and the natural complex
  structure on $i(Q)$.
\end{proposition}
Finally, Kuranishi maps preserve the stratifications in the following sense. Let
$[A,\Phi]\in\M$ such that $(A,\Phi)\in\m^{-1}(0)$. From
Section~\ref{sec:constr-moduli-space}, we see that the Kuranishi map
$\theta\colon\Z\to\B^{ss}$ induces a biholomorphism
\begin{equation}
  \varphi\colon\Z K^\C\sslash K^\C\to U\subset\M
\end{equation}
onto an open neighborhood $U$ of $[A,\Phi]$. Moreover, $\tilde{U}=\Z K^\C\sslash
K^\C$ is an open neighborhood of $[0]$ in $\nu_{0,\C}^{-1}(0)\sslash K^\C$. Now,
we can stratify $\nu_{0,\C}^{-1}(0)\sslash K^\C$ by $K^\C$-orbit types, since
there is a homeomorphism
$\nu_{0,\C}^{-1}(0)^{ps}/K^\C\to\nu_{0,\C}^{-1}(0)\sslash K^\C$. Here,
$\nu_{0,\C}^{-1}(0)^{ps}$ is the subspace of $\nu_{0,\C}^{-1}(0)$ consisting of
polystable points with respect to the $K^\C$-action. Similarly, we may stratify
the singular hyperK\"ahler quotient $\n^{-1}(0)/K$ by $K$-orbit types. By
\cite{Mayrand2018}, we obtain similar results for the stratifications on
$\nu_{0,\C}^{-1}(0)\sslash K^\C$ and $\n^{-1}(0)/K$ by replacing $\M$,
$\m^{-1}(0)/\G^\C$ and $i$ in
Proposition~\ref{sec:orbit-type-strat-main-results} and
\ref{sec:orbit-type-strat-Kob-Hit-preser} by $\nu_{0,\C}^{-1}(0)^{ps}/K^\C$,
$\n^{-1}(0)/K$ and $\n^{-1}(0)/K\xrightarrow{\sim}\nu_{0,\C}^{-1}(0)^{ps}/K^\C$,
respectively. Then, it is shown that the biholomorphism $\varphi\colon
\tilde{U}\to U$ preserves the induced stratifications on $\tilde{U}$ and $U$.
(Here, we may need to refine the induced stratifications into connected
components if strata are not connected.)

Now, we start to prove some technical results that will be used later.
\begin{proposition}\ \label{sec:orbit-type-strat-finitely-many}
  \begin{enumerate}
  \item Every $\G^\C$-stabilizer of a polystable Higgs bundle is connected.
    
  \item There are finitely many strata in $\M$.
  \end{enumerate}
\end{proposition}
\begin{proof}
  Let $(A,\Phi)$ be a polystable Higgs bundle. By definition, we may write
  \begin{equation}
    (\E_A,\Phi)=(\E_1,\Phi_1)^{\oplus m_1}\oplus\cdots\oplus(\E_r,\Phi_r)^{\oplus m_r},\qquad m_i\geq0,
  \end{equation}
  where $(\E_1,\Phi_1),\cdots,(\E_r,\Phi_r)$ are pairwise non-isomorphic stable
  Higgs bundles that have the same slope as $(\E_A,\Phi)$, and $(\E_A,\Phi)$ is
  the Higgs bundle determined by $(A,\Phi)$. As a consequence,
  \begin{equation}
    (\G^\C)_{(A,\Phi)}=\prod_{i=1}^rGL(m_i,\C).
  \end{equation}
  This proves $(1)$ and $(2)$. Here, we have used the fact that if $f$ is a
  morphism between two stable Higgs bundles of the same slope, then either
  $f\equiv0$ or $f$ is an isomorphism. Moreover, every endomorphism of a stable Higgs
  bundle must be a scalar.
\end{proof}
\begin{proposition}\label{sec:orbit-type-strat-lower-dim}
  Let $Q_1$ and $Q_2$ be two strata in $\M$. If $Q_1\subset\bar{Q_2}$, then
  $\dim Q_2>\dim Q_1$.
\end{proposition}
\begin{proof}
  Since Kuranishi maps preserve the orbit type stratifications, this problem can
  be transferred to $\nu_{0,\C}^{-1}(0)\sslash K^\C$. Write
  \begin{equation}
    \bH^1=F\oplus(\bH^1)^{K^\C},
  \end{equation}
  where $F$ is the $\omega_{0,\C}$-orthogonal complement of $(\bH^1)^{K^\C}$. By
  definition of $\nu_{0,\C}$,
  \begin{equation}
    \nu_{0,\C}^{-1}(0)=(\nu_{0,\C}|_F)^{-1}(0)\times(\bH^1)^{K^\C}
  \end{equation}
  so that
  \begin{equation}
    \nu_{0,\C}^{-1}(0)\sslash K^\C=(\nu_{0,\C}|_F)^{-1}(0)\sslash K^\C\times(\bH^1)^{K^\C}.
  \end{equation}
  Therefore, it is clear that the unique stratum containing $[0]$ is
  $(\bH^1)^{K^\C}$. If $L$ is a proper subgroup of $K^\C$, then
  \begin{equation}
    (\nu_{0,\C}^{-1}(0)\sslash K^\C)_{(L)}=((\nu_{0,\C}|_F)^{-1}(0)\sslash K^\C)_{(L)}\times(\bH^1)^{K^\C},
  \end{equation}
  where the subscript $(L)$ denote the orbit type stratum determined by $(L)$.
  As a consequence,
  \begin{equation}
    \dim(\nu_{0,\C}^{-1}(0)\sslash K^\C)_{(L)}=\dim((\nu_{0,\C}|_F)^{-1}(0)\sslash K^\C)_{(L)}+\dim(\bH^1)^{K^\C}.
  \end{equation}

  Now, we claim that if $F\neq0$ and $((\nu_{0,\C}|_F)^{-1}(0)\sslash
  K^\C)_{(L)}\neq\emptyset$, then
  \begin{equation}
    \dim((\nu_{0,\C}|_F)^{-1}(0)\sslash
  K^\C)_{(L)}>0.
  \end{equation}
  Suppose that this dimension is 0 and pick a connected component $Q$. Hence,
  $Q$ is a singleton, and its preimage in $(\nu_{0,\C}|_F)^{-1}(0)^{ps}$ is a
  single $K^\C$-orbit $xK^\C$ for some $x\neq0$. By Kempf-Ness theorem, the
  restriction of the $L^2$-norm $\|\cdot\|_{L^2}$ to the orbit $xK^\C$ attains a
  minimum value $r>0$. Therefore, we may assume that $\|x\|_{L^2}=r$. Now, we
  show that if $t_1x$ and $t_2x$ are in the same $K^\C$-orbit for some $t_1,t_2>0$, then $t_1=t_2$.
  In fact, if $t_1x=gt_2x$ for some $g\in K^\C$, then $t_1r=t_2\|gx\|_{L^2}\geq
  t_2r$ so that $t_1\geq t_2$. Applying the same argument to $g^{-1}t_1x=t_2x$,
  we obtain that $t_1\leq t_2$. Since the $K^\C$-action is linear, $tx$ is also
  polystable and has the same orbit type of $x$ for every $t\in(0,1]$.
  Therefore, $Q$ contains a subspace $\{[tx]\colon t\in(0,1]\}$, which is a
  contradiction.
\end{proof}
Using Proposition~\ref{sec:orbit-type-strat-lower-dim}, we can show that
closures of strata are closed complex subspaces of $\M$.
\begin{proposition}\label{sec:orbit-type-strat-closure-stratum}
  If $Q$ is a stratum in $\M$, then $\bar{Q}$ is a closed complex subspace of
  $\M$, and $\dim_x\bar{Q}=\dim Q$ for every $x\in\bar{Q}$.
\end{proposition}
\begin{proof}
  We prove by induction. Note that every stratum is pure dimensional. By
  Proposition~\ref{sec:orbit-type-strat-finitely-many}, let $d_1<d_2<\cdots<d_k$
  be possible values of dimensions among all the strata. By
  Proposition~\ref{sec:orbit-type-strat-lower-dim}, every stratum $Q$ of
  dimension $d_1$ is closed, and hence $\dim_x\bar{Q}=\dim Q$ for all
  $x\in\bar{Q}$. Now, suppose that the statement is true for all the strata of
  dimensions smaller than $d_i$. Let $Q$ be a stratum of dimension $d_i$.
  Therefore, $Q$ is a closed complex subspace of $\M\setminus\partial Q$. Write
  \begin{equation}
    \partial Q=Q_{l_1}\cup\cdots\cup Q_{l_k}=\bar{Q_{l_1}}\cup\cdots\cup\bar{Q_{l_k}},
  \end{equation}
  where each $Q_{l_i}$ is a stratum of dimension smaller than $d_i$. By
  induction, each $\bar{Q_{l_i}}$ is a closed complex subspace and
  $\dim_x\bar{Q_{l_i}}=\dim Q_{l_i}$ for all $x\in\bar{Q_{l_i}}$. Hence, $\dim
  Q>\dim\partial Q$. By the Remmert-Stein theorem, $\bar{Q}$ is a closed complex
  subspace. Now we show that $\dim_x\bar{Q}=\dim Q$ for every $x\in\bar{Q}$ to
  finish the proof. If $x\in Q$, then the openness of $Q$ in $\bar{Q}$ implies
  that $\dim_x\bar{Q}=\dim_xQ$. Therefore, we may assume that $x\in\partial Q$.
  Since $Q$ is open and dense in $\bar{Q}$, $\partial Q$ is nowhere dense.
  Hence, by \cite[Lemma of Ritt]{Grauert1984}, $\dim_x\partial
  Q<\dim_x\bar{Q}$. Let $S$ be an irreducible component of $\bar{Q}$ containing
  $x$ such that $\dim S=\dim_x\bar{Q}$. If $S\subset\partial Q$, then $\dim
  S\leq\dim_x\partial Q$, which is a contradiction. Hence, $S\cap
  Q\neq\emptyset$. As a consequence, since $Q$ is open in $\bar{Q}$,
  \begin{equation}
    \dim_x\bar{Q}=\dim S=\dim_x(Q\cap S)\leq\dim Q.
  \end{equation}
  Now, by the upper semicontinuity of the function $x\mapsto\dim_x\bar{Q}$ (see
  \cite[p.94]{Grauert1984}), there is an open neighborhood $U$ of $x$ in
  $\bar{Q}$ such that $\dim_y\bar{Q}\leq\dim_x\bar{Q}$ for all $y\in U$. Since
  $Q$ is open and dense in $\bar{Q}$, we may choose $y\in U\cap Q$. Hence,
  \begin{equation}
    \dim Q=\dim_y(U\cap Q)=\dim_y\bar{Q}\leq\dim_x\bar{Q}.
  \end{equation}
  Hence, $\dim Q=\dim_x\bar{Q}$.
\end{proof}
As a corollary, we obtain a codimension estimate of $\M\setminus\M^s$, where
$\M^s$ is the moduli space of stable Higgs bundles. Although this is a
well-known result (see \cite[Theorem II.6]{Faltings1993} and \cite[Lemma
11.2]{Simpson1994a}), we couldn't find an analytic proof in the literature.
\begin{corollary}\label{sec:orbit-type-strat-codimension-estimate}
  $\M^s$ is open and dense in $\M$, and $\codim_x(\M\setminus\M^s)\geq4g-6$ for
  every $x\in\M\setminus\M^s$, where $g$ is the genus of the Riemann surface $M$.
\end{corollary}
\begin{proof}
  The first statement follows from \cite[Corollary 3.24]{Stecker2015} and
  Proposition~\ref{sec:orbit-type-strat-Kob-Hit-preser}. To show the second
  statement, we write
  \begin{equation}
    \M\setminus\M^s=Q_1\cup Q_2\cup\cdots\cup Q_k=\bar{Q_1}\cup\cdots\cup\bar{Q_k},
  \end{equation}
  where each $Q_i$ is a stratum. As a consequence,
  \begin{equation}
    \dim_x(\M\setminus\M^s)=\dim_x\bar{Q_j}=\dim Q_j
  \end{equation}
  for some $j$ (depending on $x$). Therefore, by
  Proposition~\ref{sec:orbit-type-strat-closure-stratum}, we obtain
  \begin{align}
    \begin{aligned}
      \codim_x(\M\setminus\M^s)
      &=\dim_x\M-\dim_x(\M\setminus\M^s)\\
      &=\dim\M^s-\dim Q_j.
    \end{aligned}
  \end{align}
  By \cite[Corollary 3.24]{Stecker2015}
  again, $\dim\M^s-\dim Q_j\geq4g-6$. 
\end{proof}

\section{Infinite dimensional GIT quotient}\label{sec:infin-dimens-git}
\subsection{S-equivalence classes and closures of
  orbits}\label{sec:s-equiv-class}
In Section~\ref{sec:infin-dimens-git}, we will prove
Theorem~\ref{sec:introduction-infi-dim-GIT}. We start with a simple lemma in
point-set topology.
\begin{lemma}\label{sec:s-equiv-class-lemma-orbit-space}
  Let $Y$ be a first countable topological space on which a topological group
  $G$ acts. Let ${x_n}$ be a sequence in $Y$ such that $[x_n]$ converges to
  $[x]$ for some $x\in Y$ in $Y/G$. Then, there exists a subsequence ${x_{n_k}}$
  and a sequence $g_k\in G$ such that $x_{n_k}g_k$ converges to $x$ in $Y$.
\end{lemma}
\begin{proof}
  Let $\pi\colon Y\to Y/G$ be the quotient map. Since $Y$ is first countable, we
  can find nested open neighborhoods
  \begin{equation}
    \cdots\subset U_k\subset U_{k-1}\subset\cdots\subset U_1
  \end{equation}
  of $x$ that form a neighborhood basis. Now, for each $U_k$, $\pi(U_k)$ is open
  and contains $[x]$. Hence, there exists some $[x_{n_k}]\in\pi(U_k)$.
  Therefore, there exists some $g_k\in G$ such that $x_{n_k}g_k\in U_k$. Now we
  claim that $x_{n_k}g_k$ converges to $x$. Let $V$ be an open neighborhood of
  $x$. Then, $U_j\subset V$ for some $j$. If $k>j$, then $x_{n_k}g_k\in
  U_k\subset U_j\subset V$.
\end{proof}

Then, among other properties, we first show that $\pi$ identify $\G^\C$-orbits
whose closures in $\B^{ss}$ intersect. 
\begin{proposition}\label{sec:s-equiv-class-s-equiv-closure-orbit}
  Two semistable Higgs bundles are S-equivalent if and only if the closures of
  their $\G^\C$-orbits in $\B^{ss}$ intersect.
\end{proposition}
\begin{proof}
  Let $r\colon\B^{ss}\to\m^{-1}(0)$ be the retraction defined by the
  Yang-Mills-Higgs flow. Since the Yang-Mills-Higgs flow preserves
  $\G^\C$-orbits, if $(A,\Phi)$ is a semistable Higgs bundle, then $r(A,\Phi)$
  is contained in $\bar{(A,\Phi)\G^\C}$. Moreover, by \cite[Theorem
  1.4]{Wilkin2008}, $r(A,\Phi)$ is isomorphic to $Gr(A,\Phi)$. Therefore, if
  $(A_1,\Phi_1)$ and $(A_2,\Phi_2)$ are two semistable Higgs bundles that are
  $S$-equivalent, then $r(A_1,\Phi_1)$ is isomorphic to $r(A_2,\Phi_2)$.
  Therefore,
  \begin{equation}
    \bar{(A_1,\Phi_1)\G^\C}\ni r(A_1,\Phi_1)\sim_{\G^\C}r(A_2,\Phi_2)\in\bar{(A_2,\Phi_2)\G^\C},
  \end{equation}
  where $\sim_{\G^\C}$ means the equivalence relation induced by the
  $\G^\C$-action.

  Conversely, suppose that
  \begin{equation}
    (B,\Psi)\in\bar{(A_1,\Phi_1)\G^\C}\cap\bar{(A_2,\Phi_2)\G^\C}\cap\B^{ss}.
  \end{equation}
  By replacing $(B,\Psi)$ by $r(B,\Psi)$, we may assume that $(B,\Psi)$ is
  polystable. Now, $r(A_1,\Phi_1)$ is also polystable and contained in
  $\bar{(A_1,\Phi_1)\G^\C}$. Since $\bar{(A_1,\Phi_1)\G^\C}$ contains a unique
  polystable orbit (see \cite[Lemma 3.7]{Fan2020}), $r(A_1,\Phi_1)$ is
  isomorphic to $(B,\Psi)$. Similar argument shows that $r(A_2,\Phi_2)$ is
  isomorphic to $(B,\Psi)$. Since $r(A_i,\Phi_i)$ is further isomorphic to
  $Gr(A_i,\Phi_i)$ for $i=1,2$. We see that $(A_1,\Phi_1)$ and $(A_2,\Phi_2)$
  are $S$-equivalent.
\end{proof}
Using the local slice theorem
(Proposition~\ref{sec:constr-moduli-space-local-slice}) for the $\G^\C$-action,
we are able to prove that polystable orbits in $\B^{ss}$ are exactly closed
orbits.
\begin{proposition}\label{sec:s-equiv-class-closedness-polystability}
  A semistable Higgs bundle is polystable if and only if its $\G^\C$-orbit is
  closed in $\B^{ss}$.
\end{proposition}
\begin{proof}
  The same proof of \cite[Proposition 2.4(ii)]{Sjamaar1995} works.
  For the sake of completeness, we spell out the details. Let $(A,\Phi)$ be a
  semistable Higgs bundle and $r\colon\B^{ss}\to\m^{-1}(0)$ be the retraction
  defined by the Yang-Mills-Higgs flow. Since the Yang-Mills-Higgs flow
  preserves the $\G^\C$-orbits, if $(A,\Phi)\G^\C$ is closed in $\B^{ss}$, then
  obviously $r(A,\Phi)\in(A,\Phi)\G^\C$. This means that $(A,\Phi)$ is
  polystable. Conversely, assume that $(A,\Phi)$ is polystable. By the
  Hitchin-Kobayashi correspondence, we may assume that $(A,\Phi)$ lies in
  $\m^{-1}(0)$. Let $(A,\Phi)g_i$ be a sequence converging to some
  $(B,\Psi)\in\B^{ss}$. Since $(A,\Phi)$ is polystable, $r[(A,\Phi)g_i]$ is
  isomorphic to $(A,\Phi)$. By the Hitchin-Kobayashi correspondence,
  $r[(A,\Phi)g_i]\in(A,\Phi)\G$. Moreover, by continuity, $r[(A,\Phi)g_i]$
  converges to $r(B,\Psi)$. Since the $\G$-action is proper, $(A,\Phi)\G$ is
  closed, and hence $r(B,\Psi)\in(A,\Phi)\G$. On the other hand,
  $r(B,\Psi)\in\bar{(B,\Psi)\G^\C}$, and hence $(A,\Phi)\in\bar{(B,\Psi)\G^\C}$.
  By Proposition~\ref{sec:constr-moduli-space-kura-map-prop}, there is an
  $\G^\C$-invariant open neighborhood $U$ of $(A,\Phi)$ such that $\Z
  K^\C\times_{K^\C}\G^\C\to U$ is a homeomorphism. As a consequence,
  $(B,\Psi)\in U$.

  Then, it suffices to show that $(A,\Phi)\G^\C$ is closed in $U$. By the
  homeomorphism $\Z K^\C\times_{K^\C}\G^\C\to U$, it suffices to prove that if
  $[0,g_i]$ converges to $[x,g]$, then $x=0$. By
  Lemma~\ref{sec:s-equiv-class-lemma-orbit-space}, there is a subsequence
  $g_{i_k}$ and a sequence $h_k\in K^\C$ such that $(0\cdot
  h_k^{-1},h_kg_{i_k})$ converges to $(x,g)$. This immediately shows that $x=0$.
\end{proof}

The following result allows us to identify $\M$ with $\B^{ss}\sslash\G^\C$. From
now on, we will use $[\cdot]_S$ and $[\cdot]$ to denote S-equivalence classes
and isomorphism classes, respectively.
\begin{proposition}\label{sec:s-equiv-class-moduli-equiv}
  The inclusion $\B^{ps}\hookrightarrow\B^{ss}$ induces a homeomorphism
  \begin{equation}
    \B^{ps}/\G^\C\xrightarrow{\sim}\B^{ss}\sslash\G^\C.
  \end{equation}
\end{proposition}
\begin{proof}
  Let $r\colon\B^{ss}\to\m^{-1}(0)$ be the retraction defined by the
  Yang-Mills-Higgs flow. We claim that $r$ induces the inverse of the map
  \begin{equation}
    \bar{j}\colon\B^{ps}/\G^\C\xrightarrow{\sim}\B^{ss}\sslash\G^\C,
  \end{equation}
  where $j\colon\B^{ps}\hookrightarrow\B^{ss}$ is the inclusion. By definition
  of the $S$-equivalence, $r$ induces a well-defined continuous map
  \begin{equation}
    \bar{r}\colon\B^{ss}\sslash\G^\C\to\B^{ps}/\G^\C\qquad[A,\Phi]_S\mapsto[r(A,\Phi)].
  \end{equation}
  Then, if $(A,\Phi)$ is a polystable Higgs bundle,
  \begin{equation}
    \bar{r}\bar{j}[A,\Phi]=\bar{r}[A,\Phi]_S=[r(A,\Phi)].
  \end{equation}
  Since $(A,\Phi)$ is polystable, it is isomorphic to the graded object
  $Gr(A,\Phi)$ and hence to $r(A,\Phi)$. Therefore, $[r(A,\Phi)]=[A,\Phi]$.
  Conversely, if $(A,\Phi)$ is semistable, then
  \begin{equation}
    \bar{j}\bar{r}[A,\Phi]_S=\bar{j}[r(A,\Phi)]=[r(A,\Phi)]_S.
  \end{equation}
  By definition of the $S$-equivalence, $(A,\Phi)$ is $S$-equivalent to $r(A,\Phi)$,
  since $r(A,\Phi)$ is isomorphic to $Gr(A,\Phi)$. Hence,
  $[r(A,\Phi)]_S=[A,\Phi]_S$.
\end{proof}
\begin{corollary}\label{sec:s-equiv-class-unique-closed-orbit}
  Every fiber of $\pi\colon\B^{ss}\to\B^{ss}\sslash\G^\C$ contains a unique
  $\G^\C$-orbit that is closed in $\B^{ss}$. 
\end{corollary}
\begin{proof}
  This follows from Proposition~\ref{sec:s-equiv-class-closedness-polystability}
  and \ref{sec:s-equiv-class-moduli-equiv}.
\end{proof}

\subsection{$\pi$-saturated open neighborhoods}\label{sec:pi-saturated-open}
To further study the quotient map $\pi\colon\B^{ss}\to \M$, we will improve the
local slice theorem (Proposition~\ref{sec:constr-moduli-space-local-slice}) so
that the open neighborhood in $\B^{ss}$ provided by the theorem is not only
$\G^\C$-invariant but also saturated with respect to $\pi$.
\begin{lemma}\label{sec:pi-saturated-open-lemma-satura}
  Let $U$ be a $\G^\C$-invariant open subset of $\B^{ss}$. Then the following are equivalent
  \begin{enumerate}
  \item If $(A,\Phi)\in U$, then the closure of its $\G^\C$-orbit in $\B^{ss}$
    is contained in $U$.
    
  \item $U$ is $\pi$-saturated.
  \end{enumerate}
\end{lemma}
\begin{proof}
  By Proposition~\ref{sec:s-equiv-class-s-equiv-closure-orbit}, $(2)$ implies
  $(1)$. To show that $(1)$ implies $(2)$, suppose that $(B,\Phi)\in U$ and
  $(B',\Phi')\in B^{ss}$ such that $\pi(B,\Psi)=\pi(B',\Psi')$. We need to show
  that $(B',\Psi')\in U$. By
  Corollary~\ref{sec:s-equiv-class-unique-closed-orbit}, there exists a
  polystable Higgs bundle $(B'',\Psi'')$ such that
  \begin{equation}
    (B'',\Psi'')\G^\C\subset\bar{(B,\Psi)\G^\C}\cap\bar{(B',\Psi')\G^\C}\cap\B^{ss}.
  \end{equation}
  By assumption $(1)$, $(B'',\Psi'')\in U$. If $(B',\Psi')\notin U$, then
  \begin{equation}
    \bar{(B',\Psi')\G^\C}\cap U\cap\B^{ss}=\emptyset.
  \end{equation}
  This is a contradiction.
\end{proof}
\begin{proposition}\label{sec:pi-saturated-open-invariant-open-contains-sat}
  Let $(A,\Phi)$ be a polystable Higgs bundle. Then, every $\G^\C$-invariant
  open neighborhood of $(A,\Phi)$ in $\B^{ss}$ contains a $\pi$-saturated open
  neighborhood.
\end{proposition}
\begin{proof}
  Let $U$ be an $\G^\C$-invariant open neighborhood of $(A,\Phi)$ in $\B^{ss}$.
  Take a neighborhood basis $V_n$ of $[A,\Phi]_S$ in $\B^{ss}\sslash\G^\C$ such
  that $V_n\subset V_{n-1}$ for all $n\geq1$. We claim that $\pi^{-1}(V_n)$ is
  contained in $U$ for some $n$, where
  $\pi\colon\B^{ss}\to\B^{ss}\sslash\G^\C$is the quotient map. Assuming the
  contrary, we can choose a sequence $(A_n,\Phi_n)$ such that
  \begin{enumerate}
  \item $(A_n,\Phi_n)\notin U$.
    
  \item $[A_n,\Phi_n]_S$ converges to $[A,\Phi]_S$ in $\M$.
  \end{enumerate}
  Since $U$ is $\G^\C$-invariant, the closure of $(A_n,\Phi_n)\G^\C$ in
  $\B^{ss}$ is contained in $\B^{ss}\setminus U$. Then, since the closure of
  $(A_n,\Phi_n)\G^\C$ in $\B^{ss}$ contains a unique polystable orbit
  (Proposition~\ref{sec:s-equiv-class-unique-closed-orbit}), we may assume that
  each $(A_n,\Phi_n)$ is polystable. As a consequence, $[A_n,\Phi_n]$ converges
  to $[A,\Phi]$ in $\B^{ps}/\G^\C$. By
  Lemma~\ref{sec:s-equiv-class-lemma-orbit-space}, there is a subsequence
  $(A_{n_k},\Phi_{n_k})$ and a sequence $g_k\in\G^\C$ such that
  $(A_{n_k},\Phi_{n_k})\cdot g_k$ converges to $(A,\Phi)$. This is impossible, since
  $(A_{n_k},\Phi_{n_k})\cdot g_k\notin U$.
\end{proof}
\begin{theorem}\label{sec:pi-saturated-open-slice-satura}
  Let $(A,\Phi)\in\m^{-1}(0)$. If $B$ is sufficiently small, then the map
  \begin{equation}
    \bar{\theta}\colon\Z K^\C\times_{K^\C}\G^\C\to\B^{ss}\qquad[x,g]\mapsto\theta(x)g
  \end{equation}
  is a homeomorphism
  onto an $\pi$-saturated open neighborhood of $(A,\Phi)$ in $\B^{ss}$.
\end{theorem}
\begin{proof}
  Let $U$ be the open neighborhood of $(A,\Phi)$ in $\B^{ss}$ provided by
  Proposition~\ref{sec:constr-moduli-space-kura-map-prop}. By
  Proposition~\ref{sec:pi-saturated-open-invariant-open-contains-sat}, let $U'$
  be an open neighborhood of $(A,\Phi)$ in $\B^{ss}$ that is $\pi$-saturated and
  contained in $U$. Then, there is a $K^\C$-invariant open neighborhood $Q$ of
  $0$ in $\bH^1$ such that $\bar{\theta}$ maps $(\Z K^\C\cap
  Q)\times_{K^\C}\G^\C$ into $U'$. Now, let $B'$ be an open ball around $0$ in
  $\bH^1$ such that $B'\subset B\cap Q$. Then, we see that
  \begin{equation}
    (B'\cap\Z)K^\C\subset(\Z\cap Q)K^\C\subset\Z K^\C\cap Q.
  \end{equation}
  Let $\Z'=B'\cap\nu_{0,\C}^{-1}(0)=B'\cap\Z$, and we have $\Z'\subset\Z$. Let
  $U''$ be the image of $\Z' K^\C\times_{K^\C}\G^\C$ under $\bar{\theta}$. Then
  \begin{equation}
    \bar{\theta}\colon\Z' K^\C\times_{K^\C}\G^\C\to U''
  \end{equation}
  is a homeomorphism, and $U''$ is contained in $U'$.

  We prove that $U''$ is $\pi$-saturated. Let $\theta(x)g\in U''$ for some
  $x\in\Z'$ and $g\in\G^\C$. By Lemma~\ref{sec:pi-saturated-open-lemma-satura},
  we need to show that the closure of $\theta(x)g\G^\C=\theta(x)\G^\C$ in
  $\B^{ss}$ is contained in $U''$. Let $g_n$ be a sequence in $\G^\C$ such that
  $\theta(x)g_n$ converges in $\B^{ss}$. Since $U'$ is $\pi$-saturated, the limiting point is
  in $U'$, and we may assume that it is $\theta(y)h$ for some $y\in\Z$ and
  $h\in\G^\C$. Since $\bar{\theta}$ is a homeomorphism, we see that $[x,g_n]$
  converges to $[y,h]$ in $\Z K^\C\times_{K^\C}\G^\C$. By
  Lemma~\ref{sec:s-equiv-class-lemma-orbit-space}, there is a subsequence
  $g_{n_j}\in\G^\C$ and a sequence $k_j\in K^\C$ such that
  $(xk_j^{-1},k_jg_{n_j})$ converges to $(y,h)$ in $\Z K^\C\times\G^\C$.
  By \cite[Corollary 4.9]{Snow1982}, $\Z'K^\C$ is saturated with respect
  to the quotient map $\Z K^\C\to\Z K^\C\sslash K^\C$. Hence $y\in\Z'$ so that
  $[y,h]\in\Z' K^\C\times_{K^\C}\G^\C$ and $\theta(y)h\in U''$.
\end{proof}

Now we can obtain Kuranishi local models for $\B^{ss}\sslash\G^\C$ in the
following way. Fix $[A,\Phi]_S\in\B^{ss}\sslash\G^\C$ such that
$(A,\Phi)\in\m^{-1}(0)$. By Theorem~\ref{sec:pi-saturated-open-slice-satura},
the natural map
\begin{equation}
  \bar{\theta}\colon\Z K^\C\times_{K^\C}\G^\C\to U
\end{equation}
is a homeomorphism onto an $\pi$-saturated open neighborhood $U$ of $(A,\Phi)$.
By the results in Section~\ref{sec:constr-moduli-space}, $\theta$ induces a
well-defined map
\begin{equation}
  \varphi\colon\Z K^\C\sslash K^\C\to\pi(U)\subset\B^{ss}\sslash\G^\C\qquad [x]\mapsto[\theta(x)]_S,
\end{equation}
and $\pi(U)$ is an open neighborhood of $[A,\Phi]_S$ in $\B^{ss}\sslash\G^\C$.
By Proposition~\ref{sec:s-equiv-class-moduli-equiv}, we see that it is a
biholomorphism.

Moreover, we can also describe the structure sheaf of $\M$ in the following way.
\begin{proposition}\label{sec:pi-saturated-open-struc-sheaf-M}
  The structure sheaf of $\M$ is equal to $\pi_*\sO_\B^{\G^\C}$. In other words,
  for any open subset $V$ in $\M$, the natural map
  $\pi^*\colon\sO(V)\mapsto\sO(\pi^{-1}V)^{\G^\C}$ is a bijection.
\end{proposition}
\begin{proof}
  It suffices to prove the following. Let $(A,\Phi)\in\m^{-1}(0)$. By
  Theorem~\ref{sec:pi-saturated-open-slice-satura} and the remark after it, we
  see that the natural map
  \begin{equation}
    \bar{\theta}\colon\Z K^\C\times_{K^\C}\G^\C\to U
  \end{equation}
  is a homeomorphism onto an $\pi$-saturated open neighborhood of $(A,\Phi)$ in
  $\B^{ss}$. Moreover, it induces a biholomorphic map $\varphi\colon\Z
  K^\C\sslash K^\C\to\pi(U)$. By the definition of the structure sheaf of $\B$,
  we easily see that $\bar{\theta}$ is actually a biholomorphism. As a
  consequence, there is a chain of isomorphisms
  \begin{equation}
    \sO(\pi(U))\xrightarrow{\varphi^*}\sO(\Z K^\C\sslash K^\C)\xrightarrow{\pi^*}\sO(\Z K^\C)^{K^\C}\xrightarrow{\sim}\sO(\Z K^\C\times_{K^\C}\G^\C)^{\G^\C}\xrightarrow{\bar{\theta}^{-1}}\sO(U)^{\G^\C}.
  \end{equation}
  Moreover, the composition is exactly
  $\pi^*\colon\sO(\pi(U))\to\sO(U)^{\G^\C}$. 
\end{proof}
As a corollary, the quotient map $\pi\colon\B^{ss}\to\B^{ss}\sslash\G^\C$ is a
categorical quotient in the following sense.
\begin{corollary}\label{sec:pi-saturated-open-cate-quo}
  Let $Z$ be a complex space and $g\colon\B^{ss}\to Z$ a $\G^\C$-invariant
  holomorphic map. Then, $g$ induces a unique holomorphic map
  $\bar{g}\colon\M\to Z$.
\end{corollary}
\begin{proof}
  Define $\bar{g}[A,\Phi]_S=g(A,\Phi)$. By
  Proposition~\ref{sec:s-equiv-class-s-equiv-closure-orbit}, it is well-defined.
  The holomorphicity of $\bar{f}$ follows from
  Proposition~\ref{sec:pi-saturated-open-struc-sheaf-M}.
\end{proof}

\begin{proof}[Proof of Theorem~\ref{sec:introduction-infi-dim-GIT}]
  This follows from
  Proposition~\ref{sec:s-equiv-class-s-equiv-closure-orbit},~\ref{sec:s-equiv-class-closedness-polystability},
  ~\ref{sec:s-equiv-class-unique-closed-orbit},
  ~\ref{sec:pi-saturated-open-struc-sheaf-M} and
  Corollary~\ref{sec:pi-saturated-open-cate-quo}.
\end{proof}

\section{Descent lemmas for vector bundles}\label{sec:desc-lemm-vect}
In this section, we will first generalize the descent lemma for vector bundles
in \cite[Theorem 2.3]{Drezet1989} to analytic Hilbert quotients, and then
prove a similar descent lemma for the quotient map
$\pi\colon\B^{ss}\to\B^{ss}\sslash\G^\C$.

Let $G$ be a complex reductive Lie group acting holomorphically on a complex
space $X$. Suppose that $X$ admits an analytic Hilbert quotient. In other words,
there is a surjective $G$-invariant holomorphic map $\pi\colon X\to X\sslash G$
such that
\begin{enumerate}
\item $\pi$ is Stein in the sense that inverse images of Stein subspaces are
  Stein.
  
\item $\sO_{X\sslash G}=\pi_*\sO_X^G$. In other words, for every open subset $U$
  of $X\sslash G$, the map $\pi^*\colon\sO_{X\sslash
    G}(U)\to\sO_{X}(\pi^{-1}U)^G$ is an isomorphism.
\end{enumerate}
\begin{proposition}\label{sec:desc-lemm-vect-descent-analytic-quotient}
  Let $E\to X$ be a holomorphic $G$-bundle over $X$. If $G_x$ acts
  trivially on the fiber $E_x$ for every $x\in X$ whose $G$-orbit is closed,
  then there is a vector bundle $F\to X\sslash G$ such that $\pi^*F=E$.
  Moreover, $\sO(F)=\pi_*\sO(E)^G$, where $\sO(F)$ and $\sO(E)$ are the sheaves
  of holomorphic sections of $F$ and $E$, respectively.
\end{proposition}
\begin{proof}
  We closely follow the proof of \cite[Theorem 2.3]{Drezet1989}. Fix $x\in
  X$ such that $Gx$ is closed in $X$. Choose a basis $\sigma_1,\cdots,\sigma_r$
  for $E_x$. Then, we may consider the map
  \begin{equation}
    s_i\colon G/G_x\to E\qquad s_i(gG_x)=g\sigma_i.
  \end{equation}
  Since $G_x$ acts trivially on $E_x$, $s_i$ is well-defined and holomorphic.
  Now, choose an open Stein neighborhood $U$ of $\pi(x)$. Since $\pi$ is an
  analytic Hilbert quotient, $\pi^{-1}(U)$ is an open Stein neighborhood
  containing $Gx$. Since $Gx$ is closed in $\pi^{-1}(U)$, by \cite[Proposition
  3.1.1]{Heinzner1999}, $Gx$ is a closed complex subspace of
  $\pi^{-1}(U)$. Since $G$ acts transitively on $Gx$, $Gx$ is smooth. Therefore,
  the natural map $G/G_x\to Gx$ is a biholomorphism. As a consequence, we obtain
  a $G$-equivariant map
  \begin{equation}
    s_i\colon Gx\to E\qquad s_i(gx)=g\sigma_i,
  \end{equation}
  which is a holomorphic section of $E$ over $Gx$. Since $Gx$ is a closed
  complex subspace of $\pi^{-1}(U)$, and $\pi^{-1}(U)$ is Stein, $s_i$ can be
  extended to a holomorphic section of $E$ over $\pi^{-1}(U)$. By averaging over
  a maximal compact subgroup $K$ of $G$, we may assume that each $s_i$ is
  $K$-equivariant. Since $G$ is the complexification of $K$, the argument in the
  proof of \cite[Theorem 1.1]{Roberts1986} shows that each $s_i$ is also
  $G$-equivariant. Since $G_x$ acts trivially on $E_x$, $s_i(x)=\sigma_i$, and
  hence $\{s_i\}$ is linearly independent over an open neighborhood $V$ of $x$
  in $\pi^{-1}(U)$. Since each $s_i$ is $G$-equivariant, $V$ can be chosen to be
  $G$-invariant. By \cite[Proposition 3.10]{Mayrand2018}, we may further
  assume that $V=\pi^{-1}(U')$ for some smaller open neighborhood $U'\subset U$
  of $\pi(x)$ in $\M$.

  Now, note that every fiber of $\pi$ contains a unique closed orbit
  (\cite[\S3, Corollary 3]{Heinzner1998}). Therefore, the argument in
  the above paragraph provides an open covering  $U_i$ of $X\sslash G$ such that
  $E$ is trivial over $\pi^{-1}(U_i)$, and the transition functions
  $g_{ij}\colon\pi^{-1}(U_i)\cap\pi^{-1}(U_j)\to GL_n(\C)$ are $G$-invariant. As
  a consequence, by the definition of the structure sheaf of $X\sslash G$, they
  descend to holomorphic functions
  $\tilde{g}_{ij}\colon\pi(\pi^{-1}(U_i)\cap\pi^{-1}(U_j))\to GL_n(\C)$. Since
  every fiber of $\pi$ contains a unique closed orbit,
  $\pi(\pi^{-1}(U_i)\cap\pi^{-1}(U_j))=U_i\cap U_j$. Then, the data
  $\{\tilde{g}_{ij},U_i\}$ defines a holomorphic vector bundle $F$ over
  $X\sslash G$. It is easy to see that $\pi^*F=E$ and $\sO(F)=\pi_*\sO(E)^G$.
\end{proof}

\begin{proof}[Proof of Theorem~\ref{sec:introduction-descent-lemma}]
  Fix $(A,\Phi)\in\m^{-1}(0)$. By
  Theorem~\ref{sec:pi-saturated-open-slice-satura}, the map 
  \begin{equation}
    \bar{\theta}\colon\Z K^\C\times_{K^\C}\G^\C\to\B^{ss}
  \end{equation}
  induced by the Kuranishi map $\theta\colon\Z K^\C\to\B^{ss}$ for $(A,\Phi)$ is
  a $\G^\C$-equivariant homeomorphism onto a $\pi$-saturated open neighborhood
  of $(A,\Phi)$ in $\B^{ss}$. Then, we consider the pullback bundle
  $\bar{\theta}^*\bL$ on $\Z K^\C\times_{K^\C}\G^\C$. Clearly, $\theta^*\bL$ is
  the restriction of $\bar{\theta}^*\bL$ to $\Z K^\C$. By
  Proposition~\ref{sec:constr-moduli-space-kura-map-prop}, if $x\in\Z$ has a
  closed $K^\C$-orbit, then $\theta(x)$ is polystable, and
  $(K^\C)_x=(\G^\C)_{\theta(x)}$. Therefore, $(K^\C)_x$ acts trivially on the
  fiber $(\theta^*\bL)_x=\bL_{\theta(x)}$ for every $x\in\Z$ that has a closed
  $K^\C$-orbit. By
  Proposition~\ref{sec:desc-lemm-vect-descent-analytic-quotient}, the bundle
  $\theta^*\bL$ descends to $\Z K^\C\sslash K^\C$. By shrinking $\Z$ if
  necessary, we may assume that the descended bundle is trivial over $\Z
  K^\C\sslash K^\C$. As a consequence, there is a holomorphic frame
  $\{\sigma_i\}$ for $\theta^*\bL$ over $\Z K^\C$ such that each $\sigma_i$ is
  $K^\C$-equivariant. Hence, each section $\sigma_i$ extends to a
  $\G^\C$-equivariant holomorphic section of $\bar{\theta}^*\bL$ over $\Z
  K^\C\times_{K^\C}\G^\C$. Transported back to $\B^{ss}$ by $\bar{\theta}$, we
  obtain a local frame $\{\sigma_i\}$ for $\bL\to\B^{ss}$ such that each
  $\sigma_i$ is a $\G^\C$-equivariant holomorphic section over a $\pi$-saturated
  open neighborhood of $(A,\Phi)$ in $\B^{ss}$. The rest follows from the second
  paragraph in the proof of
  Proposition~\ref{sec:desc-lemm-vect-descent-analytic-quotient}.
\end{proof}

\section{The K\"ahler metric on the moduli space}\label{sec:kahler-metric-moduli}
In this section, we will prove
Theorem~\ref{sec:introduction-singular-Kahler-metric}. Let us start with the
construction of a line bundle on the moduli space $\M$. By
\cite{Donaldson1987a}, there is a holomorphic Hermitian line bundle $\bL$
over $\A$ such that the curvature of the Hermitian metric is precisely
$-2\pi\sqrt{-1}\Omega_1$, where
\begin{equation}
  \Omega_1(\alpha_1,\alpha_2)=\frac{1}{4\pi^2}\int_X\tr(\alpha_1\wedge\alpha_2),\qquad\alpha_1,\alpha_2\in\Omega^1(\g_E).
\end{equation}
Moreover, the $\G^\C$-action on $\A$ lifts to $\bL$, and the $\G$-action
preserves the Hermitian metric. The vertical part of the infinitesimal action of
$\xi\in\Omega^0(\g_E)$ on a smooth section $s$ of $\bL$ is given by
$2\pi\sqrt{-1}\langle\mu_1,\xi\rangle s$, where $\xi^\#$ is the vector field on
$\A$ generated by $\xi$, and
\begin{equation}
  \mu_1(A)=\frac{1}{4\pi^2}F_A.
\end{equation}
On the other hand, we consider the trivial line bundle
$\Omega^{1,0}(\g_E^\C)\times\C$ over $\Omega^{1,0}(\g_E^\C)$. A K\"ahler
potential on $\Omega^{1,0}(\g_E^\C)$ is given by
\begin{equation}
  \rho(\Phi)=\frac{1}{8\pi^2}\|\Phi\|_{L^2}^2.
\end{equation}
Letting $\Omega_2=\sqrt{-1}\partial\bar{\partial}\rho$, we see that
$\Omega_1+\Omega_2=\Omega_I$ on $\sC$. Let $s(\Phi)=(\Phi,1)$ be the canonical
section of the trivial line bundle $\Omega^{1,0}(\g_E^\C)\times\C$. Setting
\begin{equation}
  |s|^2=\exp(-2\pi\rho),
\end{equation}
we obtain a Hermitian metric on the trivial line bundle
$\Omega^{1,0}(\g_E^\C)\times\C$ such that its curvature is
$-2\pi\sqrt{-1}\Omega_2$. Letting $\G^\C$ act on $\sC$ trivially, we see that
the $\G^\C$-action on $\Omega^{1,0}(\g_E^\C)$ lifts to the trivial line bundle
$\Omega^{1,0}(\g_E^\C)\times\C$. Moreover, the induced $\G$-action preserves the
Hermitian metric, since $\rho$ is $\G$-invariant. Finally, the vertical part of
the infinitesimal action of $\xi\in\Omega^0(\g_E)$ on $s$ is given by
$2\pi\sqrt{-1}\langle\mu_2,\xi\rangle s$, where
\begin{equation}
  \mu_2(\Phi)=\frac{1}{4\pi^2}[\Phi,\Phi^*].
\end{equation}

Now, we pullback the line bundle $\bL$ on $\A$ and the trivial line bundle on
$\Omega^{1,0}(\g_E^\C)$ to $\sC=\A\times\Omega^{1,0}(\g_E^\C)$, and denote the
resulting line bundle still by $\bL$. We equip $\bL$ with the product of the
pullback Hermitian metrics. As a consequence, the $\G^\C$-action on $\sC$ lifts
to $\bL$, and the $\G$-action still preserves the resulting Hermitian metric.
The curvature of this Hermitian metric is precisely
\begin{equation}
  -2\pi\sqrt{-1}(\Omega_1+\Omega_2)=-2\pi\sqrt{-1}\Omega_I.
\end{equation}
Moreover, the vertical part of the infinitesimal action of
$\xi\in\Omega^0(\g_E)$ on a smooth section $s$ of $\bL$ is given by
$2\pi\sqrt{-1}\langle\mu_1+\mu_2,\xi\rangle s$, and $\mu_1+\mu_2=\mu$ on $\sC$.
The following shows that the restriction of the line bundle $\bL$ to $\B^{ss}$
descends to the moduli space $\M$.
\begin{proposition}\label{sec:kahler-metric-moduli-line-bundle-descent}
  There is a line bundle $\sL\to \M$ such that $\pi^*\sL=\bL|_{B^{ss}}$, and
  $\sO(\sL)=\pi_*\sO(\bL|_{\B^{ss}})^{\G^\C}$.
\end{proposition}
\begin{proof}
  We follow the proof of \cite[Proposition 2.14]{Sjamaar1995}. By
  Theorem~\ref{sec:introduction-descent-lemma}, it suffices to prove that
  $\G^\C_{(A,\Phi)}$ acts on $\bL_{(A,\Phi)}$ trivially for every
  $(A,\Phi)\in\m^{-1}(0)$. If $\xi\in\Lie(\G_{(A,\Phi)})$ and
  $s\in\bL_{(A,\Phi)}$, then
  \begin{equation}
    \xi\cdot s=2\pi\sqrt{-1}\langle\mu,\xi\rangle s.
  \end{equation}
  Therefore, if $\mu(A,\Phi)=0$, then $\xi\cdot s=0$. Since $\G$-stabilizers are
  connected (Proposition~\ref{sec:orbit-type-strat-finitely-many}), we conclude
  that $\G_{(A,\Phi)}$ acts trivially on $\bL_{(A,\Phi)}$. Since
  $(\G^\C)_{(A,\Phi)}$ is the complexification of $\G_{(A,\Phi)}$, we conclude
  that $\G^\C_{(A,\Phi)}$ acts trivially on $\bL_{(A,\Phi)}$.
\end{proof}

Now, we show that the moduli space $\M$ admits a weak K\"ahler metric. Let
$Q$ be a stratum in the orbit type stratification of $\M$. By
Proposition~\ref{sec:orbit-type-strat-Kob-Hit-preser}, $i^{-1}(Q)$ is a stratum
in the orbit type stratification of $\m^{-1}(0)/\G$, where
$i\colon\m^{-1}(0)/\G\xrightarrow{\sim}\B^{ps}/\G^\C$ is the Hitchin-Kobayashi
correspondence. By Proposition~\ref{sec:orbit-type-strat-main-results},
$i^{-1}(Q)$ is a hyperK\"ahler manifold. Let $\omega_{i^{-1}Q}$ be the K\"ahler
form on $i^{-1}(Q)$ induced by the K\"ahler form $\Omega_I$ on $\sC$. Then,
$(i^{-1})^*\omega_{i^{-1}Q}$ is a K\"ahler form on $Q$. Therefore, every stratum
in $\M$ is a K\"ahler manifold.

Let $[A,\Phi]_S\in\M$. By
Proposition~\ref{sec:kahler-metric-moduli-line-bundle-descent}, we may choose a
$\G^\C$-equivariant holomorphic section $s$ of $\bL$ over an $\pi$-saturated
open neighborhood $\pi^{-1}(U)$ of $(A,\Phi)$ in $\B^{ss}$ such that $s$
vanishes nowhere in $\pi^{-1}(U)$, where $U$ is an open neighborhood of
$[A,\Phi]_S$ in $\M$. Then, we define
\begin{equation}
  u=-\frac{1}{2\pi}\log|s|_h^2,
\end{equation}
where $h$ is the Hermitian metric on $\bL$. Since $h$ is preserved by the
$\G$-action, $u$ is $\G$-invariant. As a consequence, the restriction of $u$ to
$\pi^{-1}(U)\cap\m^{-1}(0)$ induces a well-defined continuous function
$u_0\colon U\to\R$.
\begin{proposition}\label{sec:kahler-metric-moduli-pluri-on-stratum}
  The function $u_0$ is continuous and smooth along each stratum $Q$. Moreover,
  $u_0|_Q$ is a K\"ahler potential for the K\"ahler form on each stratum $Q$ in
  $\M$. In particular, $u_0$ is a continuous plurisubharmonic function.
\end{proposition}
\begin{proof}
  By Proposition~\ref{sec:orbit-type-strat-main-results},
  $\pi^{-1}(Q)\cap\m^{-1}(0)\to Q$ is a submersion. Hence, $u_0$ is smooth along
  $Q$. By construction of $u$ and the Hermitian metric $h$ on $\bL$,
  \begin{equation}
    \sqrt{-1}\partial\bar{\partial}(u|_{\pi^{-1}(Q)\cap\m^{-1}(0)})=\Omega_I|_{\pi^{-1}(Q)\cap\m^{-1}(0)}.
  \end{equation}
  Therefore, the second statement follows from the construction of the K\"ahler
  form on $Q$. Now we have shown that the restriction of $u_0$ to $\M^s$ is
  strictly plurisubharmonic. Since it is continuous, by the normality of $\M$
  and the extension theorem of plurisubharmonic functions (see
  \cite{Grauert1956}), we conclude that $u_0$ is plurisubharmonic.
\end{proof}
\begin{proof}[Proof of Theorem~\ref{sec:introduction-singular-Kahler-metric}]
  By Proposition~\ref{sec:kahler-metric-moduli-pluri-on-stratum}, there is an open covering $U_i$
  of $\M$, and a stratum-wise strictly plurisubharmonic function $\rho_i\colon U_i\to\R$ on
  each $U_i$ such that $\rho_i|_{\M^s}-\rho_j|_{\M^s}$ is pluriharmonic on
  $\M^s\cap U_i\cap U_j$. Hence, we may write
  \begin{equation}
    \rho_i|_{\M^s}-\rho_j|_{\M^s}=\Re(f_{ij})
  \end{equation}
  for some holomorphic function $f_{ij}\colon U_i\cap U_j\cap\M^s\to\C$. By
  Corollary~\ref{sec:orbit-type-strat-codimension-estimate} and the normality of
  $\M$, $f_{ij}$ has a unique holomorphic extension to $U_i\cap U_j$. Then, we
  have
  \begin{equation}
    \rho_i-\rho_j=\Re(f_{ij})\text{ on }U_i\cap U_j.
  \end{equation}
  Hence, $\{U_i,\rho_i\}$ determines a weak K\"ahler metric on $\M$.
\end{proof}

\section{Projective compactification}\label{sec:proj-comp}

\subsection{Symplectic cuts}
In this section, we will use the symplectic cut to compactify the moduli space
and thus prove Theorem~\ref{sec:introduction-compactify}. Recall that there is a
holomorphic $\C^*$-action on $\sC$ given by
\begin{equation}
  t\cdot(A,\Phi)=(A,t\Phi),\qquad t\in\C^*,(A,\Phi)\in\sC.
\end{equation}
Clearly, $\B^{ss}$ is $\C^*$-invariant. Then, it is easy to verify that the
natural map $\C^*\times\B^{ss}\to\C^*\times\M$ satisfies $(4)$ in
Theorem~\ref{sec:introduction-infi-dim-GIT}, where $\G^\C$ acts on $\C^*$
trivially. Since the $\G^\C$-action and the $\C^*$-action on $\sC$ commute, we
see that the holomorphic action $\C^*\times\B^{ss}\to\B^{ss}$ descends to a
holomorphic action $\C^*\times\M\to\M$. Moreover, each stratum in the orbit
type stratification of $\M$ is $\C^*$-invariant.

Furthermore, the induced $U(1)$-action on $\M$ is stratum-wise Hamiltonian. To
see this, we first note that $U(1)$ preserves the K\"ahler form $\Omega_I$ on
$\sC$. Then, consider the function $ f\colon\sC\to\R$ given by
\begin{equation}
   f(A,\Phi)=-\frac{1}{4\pi^2}\frac{1}{2}\|\Phi\|_{L^2}^2.
\end{equation}
\begin{proposition}\label{sec:c-action-Hamiltonian-action}
  The restriction of $ f$ to $\m^{-1}(0)$ defines a continuous function, denoted
  by the same letter $f$, on $\M$ that is smooth along each stratum $Q$ of $\M$.
  Moreover, the restriction $ f|_Q$ is a moment map for the $U(1)$-action on $Q$
  with respect to the K\"ahler form on $Q$, the one induced by the K\"ahler form
  $\Omega_I$ on $\sC$.
\end{proposition}
\begin{proof}
  It is shown in \cite[p.92]{Hitchin1987b} that $ f$ is a moment map for the
  $U(1)$-action on $\sC$ with respect to the K\"ahler form $\Omega_I$. Since
  $ f$ is $\G$-invariant, its restriction to $\m^{-1}(0)$ descends to
  $\m^{-1}(0)/\G$ and hence defines a continuous function on $\M$, which we
  denote by the same letter $ f$. Let $Q$ be a stratum in the moduli space. By
  Proposition~\ref{sec:orbit-type-strat-main-results}, the restriction of $ f$
  to $\pi^{-1}(Q)\cap\m^{-1}(0)$ descends to a smooth function on $Q$ which is
  precisely the restriction of $ f\colon\M\to\R$ to $Q$. Since the quotient map
  $\pi^{-1}(Q)\cap\m^{-1}(0)\to Q$ is $U(1)$-equivariant, we conclude that
  $ f|_Q$ is a moment map for the $U(1)$-action on $Q$.
\end{proof}

To perform the symplectic cut of $\M$, we consider the direct product
$\M\times\C$ and let $\C^*$ act on $\C$ by multiplication. Hence, $\C^*$ acts
diagonally on $\M\times\C$. Moreover, $\M\times\C$ admits a stratification such
that each stratum $Q\times\C$ is equipped with the product K\"ahler form. The
next result implies that the induced $U(1)$ action on $\M\times\C$ is also
stratum-wise Hamiltonian.
\begin{proposition}\label{sec:symplectic-cuts-moment-map-on-product}
  The continuous map 
  \begin{equation}
    \tilde{ f}([A,\Phi]_S,z)= f([A,\Phi]_S)-\frac{1}{2}\|z\|^2
  \end{equation}
  is smooth along each stratum $Q\times\C$, and its restriction to $Q\times\C$
  is a moment map for the induced $U(1)$-action on $Q\times\C$ with respect to
  the product K\"ahler form on $Q\times\C$.
\end{proposition}
\begin{proof}
  It is clear that $\tilde{ f}$ is continuous on $\M\times\C$. For each stratum
  $Q$, Proposition~\ref{sec:c-action-Hamiltonian-action} implies that $ f|_Q$
  is a smooth moment map on $Q$. Since $U(1)$ acts diagonally on $Q\times\C$, it
  is easy to see that $\tilde{f}|_{Q\times\C}$ is a moment map for the
  $U(1)$-action with respect to the product K\"ahler form on $Q\times\C$.
  Therefore, $\tilde{ f}$ is a stratum-wise moment map.
\end{proof}

Now we recall the definition of the Hitchin fibration. Given a Higgs bundle
$(A,\Phi)$, the coefficient of $\lambda^{n-i}$ in the characteristic polynomial
$\det(\lambda+\Phi)$ is a holomorphic section of $\K_M^i$, where $n$ is the rank
of $\E$, $i=1,\cdots,n$, and $\K_M$ is the canonical bundle on the Riemann
surface $M$. Since these sections are clearly $\G^\C$-invariant, by
Theorem~\ref{sec:introduction-infi-dim-GIT}, we have obtained a well-defined
holomorphic map, called the Hitchin fibration,
\begin{equation}
  h\colon\M\to\bigoplus_{i=1}^nH^0(M,\K_M^i).
\end{equation}
It is known that $h$ is proper (see \cite[Theorem 2.15]{Wentworth2016} or
\cite[Theorem 8.1]{Hitchin1987b}). Therefore, the nilpotent cone $h^{-1}(0)$ is
compact so that $ f$ has a lower bound on $h^{-1}(0)$. We choose a constant
$c<0$ such that $h^{-1}(0)\subset f^{-1}(c,0]$. In other words,
$ f^{-1}(-\infty,c]$ does not contain the nilpotent cone. Then, we perform the
symplectic cut of $\M$ at the level $c$. By definition, it is the singular
symplectic quotient
\begin{equation}
  \tilde{ f}^{-1}(c)/U(1)=\Bigl\{([A,\Phi]_S,z)\in\M\times\C\colon f([A,\Phi]_S)-\frac{1}{2}\|z\|^2
  =c\Bigr\}/U(1).
\end{equation}
If $\M\times\C$ admits a (strong) K\"ahler metric, then we may directly apply
the analytic GIT developed in \cite{Heinzner1994}. Since we are unable to prove
this, we will have to take a detour to prove that the symplectic cut of $\M$ at
the level $c$ is a compact complex space.

Let $W=(\M\times\C)\setminus(h^{-1}(0)\times\{0\})$. It is clear that $W$ is
$\C^*$-invariant and open. We first show that the analytic Hilbert quotient
$W/\C^*$ exists.
\begin{lemma}\label{sec:symplectic-cuts-properness-on-M}
  The $\C^*$-action on $\M\setminus h^{-1}(0)$ is proper.
\end{lemma}
\begin{proof}  Clearly, $h^{-1}(0)$ is $\C^*$-invariant. Suppose that
  \begin{enumerate}
  \item $x_i$ converges to $x'\notin h^{-1}(0)$.
  \item $t_i\cdot x_i$ converges to $y\notin h^{-1}(0)$.
  \end{enumerate}
  We first claim that $|t_i|$ cannot be unbounded. If not, we may assume that
  $|t_i|\to\infty$ and let $t_i'=t_i/|t_i|$. By passing to a subsequence, we may
  assume that $t_i'$ converges to $t_\infty'$, and $|t_\infty'|=1$. As a
  consequence, since $x'\notin h^{-1}(0)$,
  \begin{equation}
    \lim_{i\to\infty}t'_i\cdot x_i=t_\infty'\cdot x'\notin h^{-1}(0).
  \end{equation}
  On the other hand, since $t_i\cdot x_i$ converges,
  \begin{equation}
    \lim_{i\to\infty}h\biggl(\frac{1}{|t_i|}t_i\cdot x_i\biggr)=0.
  \end{equation}
  Since $h$ is proper, by passing to a subsequence, we may assume that
  \begin{equation}
    \frac{1}{|t_i|}t_i\cdot x_i=t_i'\cdot x_i
  \end{equation}
  converges to an element in $h^{-1}(0)$. This is a contradiction.

  Since $t_i$ is bounded, it has a subsequence that is convergent. We claim that
  such a sequence cannot converge to 0. If not, suppose that $t_i\to0$. Then,
  \begin{equation}
    \lim_{i\to\infty}h(t_i\cdot x_i)=0,
  \end{equation}
  and hence $t_i\cdot x_i$ has a subsequence converging to an element in
  $h^{-1}(0)$. Therefore, $y\in h^{-1}(0)$. This is a contradiction.
\end{proof}
\begin{corollary}\label{sec:symplectic-cuts-properness}
  The $\C^*$-action on $W$ is proper.
\end{corollary}
\begin{proof}  Note that
  \begin{equation}
    W=(\M\setminus h^{-1}(0)\times\C)\cup(\M\times\C^*).
  \end{equation}
  Suppose there are sequences
  \begin{enumerate}
  \item $(x_i,a_i)\in W$ converges to $(x',a')\in W$
  \item $t_i\cdot(x_i,a_i)\in W$ converges to $(y,b)\in W$
  \end{enumerate}
  We need to show that $t_i$ has a subsequence that converges in $\C^*$. We
  prove this by considering the following cases:
  \begin{enumerate}
  \item Suppose $(x',a')\in\M\times\C^*$. Then, $(x_i,a_i)\in\M\times\C^*$ if
    $i\gg0$. Hence, $t_i=(t_ia_i)a_i^{-1}$ converges to $ba'^{-1}$. If $b\neq0$,
    then $ba'^{-1}\in\C^*$, and we are done with this case. If $b=0$, then
    $y\notin h^{-1}(0)$. Moreover, $\lim_{i\to\infty}h(t_ix_i)=0$. Then,
    $t_ix_i$ has a subsequence converging to an element in $h^{-1}(0)$. Since
    this element has to be $y$, we have shown that $b=0$ is impossible.

  \item Suppose that $(x',a')\in(\M\setminus h^{-1}(0))\times\C$. Then, both
    $(x_i,a_i)$ and $t_i\cdot(x_i,a_i)$ lie in $(\M\setminus h^{-1}(0))\times\C$
    if $i\gg0$. If $y\notin h^{-1}(0)$, then
    Lemma~\ref{sec:symplectic-cuts-properness-on-M} applies. Hence, we may
    assume that $y\in h^{-1}(0)$ and hence $b\neq0$. If $a'\neq0$, then
    $t_i=(t_ia_i)a_i^{-1}$ converges to $ba'^{-1}$, and we are done. Hence, we
    may assume that $a'=0$, and therefore
    \begin{align}
      \begin{aligned}
        t_ia_i&\to b\neq0,\\
        a_i&\to a'=0.
      \end{aligned}
    \end{align}
    We claim that $t_i$ is bounded, so that $t_ia_i$ converges to 0, which is a
    contradiction. If not, we may assume that $|t_i|\to\infty$ and let
    $t_i'=t_i/|t_i|$. By passing to a subsequence, we may further assume that
    $t_i'$ converges to $t_\infty'$ with $|t_\infty'|=1$. As a consequence,
    $t'_ix_i$ converges to $t'_\infty x'\notin h^{-1}(0)$. On the other hand,
    \begin{equation}
      \lim_{i\to\infty}h(t'_ix_i)=\lim_{i\to\infty}h\biggl(\frac{1}{|t_i|}t_ix_i\biggr)=0.
    \end{equation}
    Hence, the properness of $h$ implies that $t'_ix_i$ contains a subsequence
    converging to an element in $h^{-1}(0)$, which is a contradiction.
  \end{enumerate}
\end{proof}
\begin{corollary}
  The analytic Hilbert quotient of $W$ by $\C^*$ exists. Moreover, $W/\C^*$ is a
  geometric quotient.
\end{corollary}
\begin{proof}
  This follows from Corollary~\ref{sec:symplectic-cuts-properness} and \cite[\S4, Corollary 2]{Heinzner1998}.
\end{proof}

Then, we study the relationship between the symplectic cut
$\tilde{ f}^{-1}(c)/U(1)$ and the analytic Hilbert quotient $W/\C^*$. Let
$(\M\times\C)^{ss}$ be the semistable points in $\M\times\C$ determined by
$\tilde{ f}-c$. In other words, $([A,\Phi]_S,z)$ lies in $(\M\times\C)^{ss}$ if
and only if the closure of its $\C^*$-orbit in $\M\times\C$ intersects
$\tilde{ f}^{-1}(c)$.
\begin{lemma}\label{sec:symplectic-cuts-W=ss}
  $W=(\M\times\C)^{ss}=\C^*\cdot\tilde{ f}^{-1}(c)$.
\end{lemma}
\begin{proof}
  We first show that $\tilde{ f}^{-1}(c)\subset W$. Suppose that this is not
  true, and we choose some $([A,\Phi]_S,z)\in\tilde{ f}^{-1}(c)$ such that
  $([A,\Phi]_S,z)\notin W$. In other words, $[A,\Phi]_S\in h^{-1}(0)$ and $z=0$.
  Hence, $\tilde{ f}([A,\Phi]_S,z)= f([A,\Phi]_S)=c$. This cannot happen by
  the choice of the level $c$.

  Then, we show that $(\M\times\C)^{ss}\subset W$. If the closure of the
  $\C^*$-orbit of a point $([A,\Phi]_S,z)$ in $\M\times\C$ meets
  $\tilde{ f}^{-1}(c)$, then it must meet $W$, since $W$ is open. Since $W$ is
  also $\C^*$-invariant, $W$ contains $([A,\Phi]_S,z)$.

  Finally, we show that $W\subset(\M\times\C)^{ss}$. For every $([A,\Phi]_S,z)\in
  W$, consider the function
  \begin{equation}
    q(t)= f([A,t\Phi])-\frac{1}{2}t^2\|z\|^2-c\qquad t>0.
  \end{equation}
  We show that $q(t_0)=0$ for some $t_0>0$ so that $t\cdot([A,\Phi]_S,z)$ lies
  in $\tilde{ f}^{-1}(c)$ for some $t>0$. Since $h([A,t\Phi])\to0$ as $t\to0$,
  the properness of $h$ implies that there exists a sequence ${t_n}\subset\C^*$
  such that $t_n\to0$ and $[A,t_n\Phi]$ converges to some $[B,\Psi]_S\in
  h^{-1}(0)$. Hence, letting $n\to\infty$, we see that
  \begin{equation}
    \lim_{n\to\infty}q(t_n)= f([B,\Psi]_S)-c>0.
  \end{equation}
  On the other hand, since $ f\leq0$, we have
  \begin{equation}
    q(t)\leq-\frac{1}{2}t^2\|z\|^2-c.
  \end{equation}
  If $z\neq0$, then $t\gg0$ implies that $q(t)<0$. Hence, $q(t_0)=0$ for some
  $t_0>0$.

  Now, we assume that $z=0$ so that $[A,\Phi]_S\notin h^{-1}(0)$. We claim that
  the function $t\mapsto f([A,t\Phi]_S)$ is unbounded below as $t\to\infty$. If
  this claim is true, then $q(t)<0$ if $t\gg0$, and hence $q(t_0)=0$ for some
  $t_0>0$. Now, we prove the claim. Assuming the contrary, we may choose a
  sequence of $\{t_n\}\subset\C^*$ such that $t_n\to\infty$ and $
  f([A,t_n\Phi])$ is bounded. By the properness of $ f$, by passing to a
  subsequence, we may assume that $[A,t_n\Phi]$ converges to some $[B,\Psi]_S$.
  Hence, $h([A,t_n\Phi])$ also converges as $t_n\to\infty$. This implies that
  $h([A,\Phi]_S)=0$, which is a contradiction.

  Finally, note that the proof has already shown that
  $(\M\times\C)^{ss}=\C^*\cdot\tilde{ f}^{-1}(0)$.
\end{proof}
\begin{corollary}\label{sec:symplectic-cuts-symplectic-quotient}
  The inclusion $\tilde{ f}^{-1}(c)\hookrightarrow(\M\times\C)^{ss}$ induces a
  homeomorphism
  \begin{equation}
    \tilde{ f}^{-1}(c)/U(1)\xrightarrow{\sim}(\M\times\C)^{ss}/\C^*=W/\C^*.
  \end{equation}
  Moreover, $W/\C^*$ is compact.
\end{corollary}
\begin{proof}
  Since $ f$ and the norm $\|\cdot\|$ on $\C$ are proper, $ f^{-1}(c)$ is
  compact. Therefore, $\tilde{ f}^{-1}(c)/U(1)$ is also compact. Moreover,
  since $(\M\times\C)^{ss}/\C^*$ is Hausdorff, to show that the map is a
  homeomorphism, it suffices to show that it is a continuous bijection. The
  continuity is obvious. By Lemma~\ref{sec:symplectic-cuts-W=ss}, the
  surjectivity is clear.

  To show the injectivity, suppose that $([A_1,\Phi_1],z_1)$ and
  $([A_2,\Phi_2],z_2)$ lie in $\tilde{ f}^{-1}(c)$ and the same $\C^*$-orbit.
  Since each orbit type stratum in $\M$ is $\C^*$-invariant, they lie in
  $Q\times\C$ for some stratum $Q$ in $\M$. By
  Proposition~\ref{sec:symplectic-cuts-moment-map-on-product},
  $\tilde{ f}|_{Q\times\C}$ is a moment map for the $U(1)$-action on
  $Q\times\C$ with respect to the product K\"ahler form on $Q\times\C$. Hence,
  $([A_1,\Phi_1],z_1)$ and $([A_2,\Phi_2],z_2)$ must lie in the same
  $U(1)$-orbit by general properties of moment maps (see \cite[Lemma
  7.2]{Kirwan1984}).
\end{proof}
\begin{proof}[Proof of Theorem~\ref{sec:introduction-compactify}]
  Write $W=(\M\setminus h^{-1}(0)\times\{0\})\cup(\M\times\C^*)$. Note that it is
  a disjoint union. Let $W^*=\M\times\C^*$ and consider the map
  \begin{equation}
    W^*\to\M,\qquad ([A,\Phi]_S,z)\mapsto z^{-1}[A,\Phi]_S.
  \end{equation}
  Since it is $\C^*$-invariant, it induces a well-defined map
  $(\M\times\C^*)/\C^*\to\M$. The injectivity is clear. Its inverse is given by
  $[A,\Phi]_S\mapsto([A,\Phi]_S,1)$.

  Then, we show that it is a biholomorphism. Since $\M$ is normal, $\M\times\C$
  is also normal. Therefore, both $W$ and $W^*$ are normal. As categorical
  quotients of normal spaces, $W^*/\C^*$ and $W/\C^*$ are also normal. Moreover, fibers of $W^*\to
  W^*/\C^*$ have pure dimension 1. Since $\M^s$ is pure dimensional,
  Proposition~\ref{sec:orbit-type-strat-closure-stratum} implies that $\M$ and
  hence $W^*$ are pure dimensional. Therefore, by Remmert's rank theorem (see
  \cite[Proposition 1.21]{Andreotti1971}), we conclude that $W^*/\C^*$
  is pure dimensional, and
  \begin{equation}
    \dim W^*/\C^*=\dim W^*-\dim\C^*=\dim\M.
  \end{equation}
  Then, by \cite[p.166, Theorem]{Grauert1984}, the map $W^*/\C^*\to\M$
  is a biholomorphism.

  Since $W/\C^*$ is compact, we have shown that $\M$ admits a compactification
  \begin{equation}
    W/\C^*=\M\cup Z,
  \end{equation}
  where $Z=(\M\setminus h^{-1}(0)\times\{0\})/\C^*$ is of pure codimension 1.
\end{proof}

Finally, we prove the following result that will be used later. Note that $W$
inherits a stratification from $\M\times\C$. More precisely, $W$ is a disjoint
union of $Q_W=W\cap(Q\times\C)$ where $Q$ ranges in the stratification of $\M$.
Moreover, if necessary, we may also refine this stratification into connected
components. The following shows that how $Q_W/\C^*$ fits together in $\bar{\M}$.
\begin{proposition}\ \label{sec:symplectic-cuts-stratum-in-Mbar}
  Let $\pi\colon W\to \bar{\M}$ be the quotient map.
  \begin{enumerate}
  \item Each $\pi(\bar{Q_W})$ is a closed complex subspace of $\bar{\M}$, where
    the closure is taken in $W$.

  \item Each $\pi(Q_W)$ is a locally closed complex subspace of $\bar{\M}$, and
    its closure is precisely $\pi(\bar{Q_W})$.

  \item If $\pi(Q_W)\cap\bar{\pi(S_W)}\neq\emptyset$, then
    $\pi(Q_W)\subset\bar{\pi(S_W)}$.

  \item $\bar{\M}$ is a disjoint union of $\pi(Q_W)$.
    
  \item The restriction $\pi\colon Q_W\to \pi(Q_W)$ is the analytic Hilbert
    quotient of $Q_W$ by $\C^*$. Moreover, the inclusion
    $(\tilde{ f}|_{Q_W})^{-1}(c)\hookrightarrow Q_W$ induces a homeomorphism
    \begin{equation}
      (\tilde{ f}|_{Q_W})^{-1}(c)/U(1)\xrightarrow{\sim} Q_W/\C^*.
    \end{equation}
  \end{enumerate}
\end{proposition}
\begin{proof}
  Fix a stratum $Q$ in $\M$. By
  Proposition~\ref{sec:orbit-type-strat-closure-stratum}, $\bar{Q_W}$ is a
  closed complex subspace of $W$ that is also $\C^*$-invariant. Therefore,
  \cite[\S1(ii)]{Heinzner1998} implies that $\pi(\bar{Q_W})$ is a closed complex
  subspace of $\bar{M}$. Moreover, the restriction
  $\pi\colon\bar{Q_W}\to\pi(\bar{Q_W})$ is also an analytic Hilbert quotient.
  This proves $(1)$.

  Since $Q_W$ is open in $\bar{Q_W}$, and $\pi\colon \bar{Q_W}\to\pi(\bar{Q_W})$ is an
  open map, $\pi(Q_W)$ is open in $\pi(\bar{Q_W})$. Moreover, the continuity of
  $\pi$ shows that $\pi(\bar{Q_W})\subset\bar{\pi(Q_W)}$. Since $\pi(\bar{Q_W})$
  is closed in $\bar{M}$, we have $\pi(\bar{Q_W})=\bar{\pi(Q_W)}$. This proves
  $(2)$.

  If $\pi(Q_W)\cap\bar{\pi(S_W)}\neq\emptyset$ for some stratum $S$ in $\M$,
  then $\pi(Q_W)\cap\pi(\bar{S_W})\neq\emptyset$. Since $\pi\colon W\to\bar{M}$
  is a geometric quotient, and both $Q_W$ and $\bar{S_W}$ are $\C^*$-invariant,
  we conclude that $Q_W\cap\bar{S_W}\neq\emptyset$. Therefore,
  $Q_W\subset\bar{S_W}$, and hence $\pi(Q_W)\subset\bar{\pi(S_W)}$. This shows
  $(3)$.

  Obviously, $\bar{\M}$ is a union of $\pi(Q_W)$ as $Q$ ranges in the
  stratification of $\M$. Since each $Q_W$ is $\C^*$-invariant, and $\pi\colon
  W\to\bar{\M}$ is a geometric quotient, it is a disjoint union. This proves
  $(4)$.

  Finally, $(5)$ immediately follows from
  Corollary~\ref{sec:symplectic-cuts-symplectic-quotient} and the fact that
  $Q_W$ is $\C^*$-invariant.
\end{proof}

\subsection{Projectivity}
In this section, we will prove $(2)$ and $(3)$ in
Theorem~\ref{sec:introduction-positivity}. Let us start with the construction of
a line bundle on $\bar{\M}=W/\C^*$. Note that the $\C^*$-action on $\M$ lifts to
the line bundle $\sL\to\M$. This can be seen as follows. The $\C^*$-action on
$\Omega^{1,0}(\g_E^\C)$ lifts to the trivial line bundle by letting $\C^*$ act
on the fiber trivially. By construction of the line bundle $\bL\to\sC$, the
$\C^*$-action on $\sC$ lifts to $\bL$. Since the $\G^\C$-action and the
$\C^*$-action commutes, we see that the $\C^*$-action on $\bL$ descends to $\sL$
which covers the $\C^*$-action on $\M$. Then, consider the trivial line bundle
over $\C$. The $\C^*$-action on $\C$ lifts to the trivial line bundle by letting
$\C^*$ act on the fiber trivially. Moreover, we equip the trivial line bundle
with a Hermitian metric determined by
\begin{equation}
  |s|^2=\exp(-2\pi\chi),
\end{equation}
where $s(z)=(z,1)$ is a section of the trivial line bundle, and
$\chi(z)=\frac{1}{2}\|z\|^2$ is a K\"ahler potential for the standard K\"ahler
form on $\C$.

Now, we pullback the trivial line bundle on $\C$ and the line bundle $\sL\to\M$
to $\M\times\C$, and denote the resulting line bundle by $\sL_\C$. Moreover, we
equip the line bundle $\sL_\C\to\M\times\C$ with the product of the pullback
Hermitian metrics, and the $\C^*$-action on $\M\times\C$ lifts to $\sL_\C$. We
will still use the letter $h$ to denote the resulting Hermitian metric on
$\sL_\C$.
\begin{proposition}
  The Hermitian metric $h$ on $\sL_\C$ is smooth along $Q\times\C$ for every stratum
  $Q$ in $\M$. Moreover, its curvature on $Q\times\C$ is precisely $-2\pi\sqrt{-1}\omega_{Q\times\C}$.
\end{proposition}
\begin{proof}
  By Proposition~\ref{sec:kahler-metric-moduli-pluri-on-stratum}, the curvature
  of the Hermitian metric on $\sL$ along $Q$ is $-2\pi\sqrt{-1}\omega_Q$. By the
  construction of $h$ on $\sL_\C$, we see that the curvature of $h$ along
  $Q\times\C$ must be $-2\pi\sqrt{-1}\omega_{Q\times\C}$.
\end{proof}

By construction, if $p=([A,\Phi]_S,z)\in\M\times\C$, then $(\C^*)_p$ acts
trivially on $(\sL_\C)_p$. As a consequence, we obtain the following.
\begin{proposition}\label{sec:projectivity-descent}
  The canonical line bundle $\sL_\C\to\M\times\C$ descends to $\bar{\M}$. In other
  words, there is a line bundle $\bar{\sL}\to \bar{\M}$ such that $\pi^*\bar{\sL}=\sL_\C|_W$,
  where $\pi\colon W\to W/\C^*=\bar{\M}$ is the quotient map.
\end{proposition}
\begin{proof}
  This follows from
  Proposition~\ref{sec:desc-lemm-vect-descent-analytic-quotient}.
\end{proof}
Since the Hermitian metric $h$ on $\sL_\C$ is preserved by the $U(1)$-action,
Corollary~\ref{sec:symplectic-cuts-symplectic-quotient} implies that $h$ induces
a continuous Hermitian metric $\bar{h}$ on $\bar{\sL}\to\bar{\M}$. Although $h$ is
smooth along $Q_W$ for each stratum $Q$ in $\M$, $\bar{h}$ may not be smooth along
$\pi(Q_W)$. That said, we note that $Q_W$ is a K\"ahler manifold with a
$\C^*$-action such that the induced $U(1)$-action is Hamiltonian with respect to
the K\"ahler form $\omega_{Q_W}$ on $Q_W$. Hence, by \cite[Theorem
2.10]{Sjamaar1995} and Proposition~\ref{sec:symplectic-cuts-stratum-in-Mbar}, we
may further stratify $\pi(Q_W)$ by $\C^*$-orbit types. Since the curvature of
$(\sL_\C,h)$ on $Q_W$ is $-2\pi\sqrt{-1}\omega_{Q_W}$, by \cite[Lemma
2.16]{Sjamaar1995}, we conclude the following.
\begin{proposition}\label{sec:projectivity-curvature-h0}
  For each stratum $Q$ in $\M$, $\pi(Q_W)$ admits a $\C^*$-orbit type
  stratification such that each stratum $S$ is a locally closed K\"ahler
  submanifold of $\pi(Q_W)$ with K\"ahler form $\omega_S$. Moreover, the
  Hermitian metric $\bar{h}$ on $\bar{\sL}$ is smooth along $S$, and the its curvature
  on $S$ is precisely $-2\pi\sqrt{-1}\omega_S$.
\end{proposition}

Now we are ready to prove that the line bundle $\bar{\sL}\to\bar{\M}$ is ample. The
first step is the following.
\begin{lemma}
  The Chern current $c_1(\bar{\sL},\bar{h})$ of $(\bar{\sL},\bar{h})$ is positive, where
  \begin{equation}
    c_1(\bar{\sL},\bar{h})=\frac{\sqrt{-1}}{2\pi}\bar{\partial}\partial\log|s|_{\bar{h}}^2,
  \end{equation}
  and $s$ is any local holomorphic section of $\bar{\sL}$ that is nowhere vanishing.
\end{lemma}
\begin{proof}
  By Proposition~\ref{sec:projectivity-descent}, for every open subset $U$ of
  $\bar{\M}$, we may choose a $\C^*$-equivariant holomorphic section $s$ of
  $\sL_\C$ over $\pi^{-1}(U)$ that is nowhere vanishing, where $\pi\colon
  W\to\bar{\M}$ is the quotient map. Then, we define
  \begin{equation}
    v=-\frac{1}{2\pi}\log|s|_h^2.
  \end{equation}
  Since $v$ is $U(1)$-invariant, by
  Corollary~\ref{sec:symplectic-cuts-symplectic-quotient}, the restriction of
  $v$ to $\pi^{-1}(U)\cap\tilde{ f}^{-1}(c)$ induces a well-defined continuous
  function $v_0\colon U\to\R$. If $Q=\M^s$, then
  Proposition~\ref{sec:symplectic-cuts-stratum-in-Mbar} and
  \ref{sec:projectivity-curvature-h0} imply that $\pi(Q_W)$ is open in
  $\bar{\M}$ and that $\pi(Q_W)$ admits a $\C^*$-orbit type stratification.
  Moreover, if $S$ is the top-dimensional stratum, then $S$ is open and dense in
  $\pi(Q_W)$. By Proposition~\ref{sec:projectivity-curvature-h0} again, the
  restriction of $v_0$ to $S$ is a K\"ahler potential for the K\"ahler form on
  $S$ so that $v_0|_S$ is strictly plurisubharmonic. Since $v_0$ is already
  continuous, and $\bar{\M}$ is normal, the extension theorem of
  plurisubharmonic functions (see~\cite{Grauert1956}) implies that $v_0\colon
  U\to\R$ is plurisubharmonic. Since
  $c_1(\bar{\sL},\bar{h})=\sqrt{-1}\partial\bar{\partial} v_0$, we see that
  $c_1(\bar{\sL},\bar{h})$ is positive.
\end{proof}
Then, the key result to show that $\bar{\sL}$ is ample is the following.
\begin{lemma}\label{sec:projectivity-bigness}
  For every closed irreducible complex subspace $Y$ of $\bar{\M}$ with $\dim
  Y>0$, the restriction of the line bundle $\bar{\sL}\to\bar{\M}$ to $Y$ is big.
\end{lemma}
\begin{proof}
  By $(3)$ in Proposition~\ref{sec:symplectic-cuts-stratum-in-Mbar}, there is a
  natural partial order among $\pi(Q_W)$, where $Q$ ranges in the stratification
  of $\M$. We define $\pi(Q_W)\leq\pi(S_W)$ if $\pi(Q_W)\subset\bar{\pi(S_W)}$.
  If $Y$ is a closed irreducible complex subspace, $Y$ must intersect some
  $\pi(Q_W)$. We choose $\pi(Q_W)$ to be the largest one with respect to the
  partial order $\leq$ just mentioned. By $(3)$ in
  Proposition~\ref{sec:symplectic-cuts-stratum-in-Mbar} again, $\pi(Q_W)$ is
  open in $\bar{\M}\setminus\cup_{\pi(S_W)>\pi(Q_W)}\pi(S_W)$. Therefore,
  $\pi(Q_W)\cap Y$ is open in $Y$. By
  Proposition~\ref{sec:projectivity-curvature-h0}, $\pi(Q_W)$ admits a
  $\C^*$-orbit type stratification. Similarly, we may further choose a stratum
  $S$ in $\pi(Q_W)$ such that $Y\cap\pi(Q_W)\cap S=Y\cap S$ is open in
  $Y\cap\pi(Q_W)$. Therefore, $Y_{reg}\cap S$ is also open in $Y$, where
  $Y_{reg}$ is the smooth locus of $Y$.

  Now we consider the restriction of $\bar{\sL}$ to $Y$. We will use \cite[Theorem
  1.3]{Popovici2008} to show that $\bar{\sL}|_Y$ is big. By taking a
  desingularization of $Y$, we may assume that $Y$ is a compact complex
  manifold. Clearly, the Chern current $c_1(\bar{\sL},\bar{h})$ is still positive.
  Therefore, by Lebesgue's decomposition theorem, the absolutely continuous part
  $c_1(\bar{\sL},\bar{h})_{ac}$ of $c_1(\bar{\sL},\bar{h})$ is also positive. Hence,
  \begin{equation}\label{eq:1}
    \int_Yc_1(\bar{\sL},\bar{h})_{ac}^{\dim Y}\geq\int_{Y_{reg}\cap S}c_1(\bar{\sL},\bar{h})_{ac}^{\dim Y}>0
  \end{equation}
  To justify the last inequality, we note that $Y_{reg}\cap S$ is a complex
  submanifold of $S$ and hence K\"ahler. By
  Proposition~\ref{sec:projectivity-curvature-h0}, $c_1(\bar{\sL},\bar{h})$ is the
  K\"ahler form on $S$. Therefore, the restriction of $c_1(\bar{\sL},\bar{h})_{ac}^{\dim
    Y}$ to $Y_{reg}\cap S$ is precisely the volume form on $Y_{reg}\cap S$.
  Hence, the last inequality in Equation~\eqref{eq:1} holds.
\end{proof}
\begin{proof}[Proof of Theorem~\ref{sec:introduction-positivity}]
  Note that $(1)$ in Theorem~\ref{sec:introduction-positivity} is already proved
  in Proposition~\ref{sec:kahler-metric-moduli-line-bundle-descent}. To prove
  $(3)$, we use Grauert's criterion of ampleness for a line bundle over a
  compact complex space (see \cite{Grauert1962}). Therefore, we need to show
  that the restriction of $\bar{\sL}$ to any irreducible closed complex subspace $Y$
  with $\dim Y>0$ admits a nontrivial holomorphic section that vanishes
  somewhere on $Y$. Let $Y$ be an irreducible closed complex subspace of
  $\bar{\M}$ with $\dim Y>0$. By Lemma~\ref{sec:projectivity-bigness},
  $\bar{\sL}|_Y$ is big. Hence, it admits a nontrivial holomorphic section. Such a
  section must vanish somewhere on $Y$. Otherwise, $\bar{\sL}|_Y$ is holomorphically
  trivial and cannot be big. Therefore, $\bar{\sL}$ is ample, and $\bar{\M}$ is
  projective.

  To see that $\M$ is quasi-projective, let us recall that $\bar{\M}=\M\cup Z$,
  where $Z$ is a closed complex subspace of $\bar{\M}$. Moreover, let
  $i\colon\bar{\M}\to\mathbb{P}^N$ be a projective embedding. By Remmert's
  proper mapping theorem, $i(Z)$ is a closed complex subspace of $\mathbb{P}^N$.
  By Chow's theorem, both $i(\bar{\M})$ and $i(Z)$ are Zariski closed in
  $\mathbb{P}^N$ so that $i(\M)$ is Zariski open in $i(\bar{\M})$. By
  definition, $\M$ is quasi-projective.

  Finally, we show $(2)$. It suffices to show that the line bundle $\bar{\sL}\to
  W^*/\C^*$ is isomorphic to $\sL\to\M$ via the biholomorphism $W^*/\C^*\to\M$
  described in the proof of Theorem~\ref{sec:introduction-compactify}. By
  definition, the total space of the line bundle $\sL_\C\to\M\times\C$ is
  $\sL\times\C$. If we restrict $\sL_\C$ to $W^*=\M\times\C^*$, we obtain the
  following commutative diagram
  \begin{equation}
    \label{eq:2}
    \xymatrix
    {
      \sL\times\C^*\ar[r]\ar[d] &\sL\ar[d]\\
      \M\times\C^*\ar[r] &\M
    },
  \end{equation}
  where the top horizontal map is given by $(v,z)\mapsto z^{-1}\cdot v$.
  Therefore, the diagram~\eqref{eq:2} defines a map $(\sL_\C|_{W^*})/\C^*\to\sL$
  covering the biholomorphism $W^*/\C^*\to\M$. Finally, by the proof of
  Proposition~\ref{sec:desc-lemm-vect-descent-analytic-quotient}, it is easy to
  verify that the total space $\bar{\sL}$ of the line bundle $\bar{\sL}\to W^*/\C^*$ is
  precisely $(\sL_\C|_{W^*})/\C^*$. Therefore, we have obtained a bundle map
  $\bar{\sL}\to\sL$ that is an isomorphism on each fiber.
\end{proof}

\bibliography{..//..//references}
\bibliographystyle{abbrv}
\end{document}